\documentclass[letterpaper,oneside, reqno]{amsart}
\usepackage{mathtools, graphicx, amsthm, amsfonts, amssymb}
\usepackage{breqn}
\usepackage[pdftex]{pict2e}
\usepackage[dvipsnames]{xcolor}
\usepackage[utf8]{inputenc}
\usepackage{times}
\usepackage[T1]{fontenc}
\usepackage{amsrefs}
\usepackage{microtype}
\usepackage[letterpaper]{geometry}
\usepackage{hyperref}
\usepackage{cleveref}

\usepackage{fancyhdr}
\usepackage[draft]{fixme}
\fxsetup{theme=color, mode=multiuser, noinline}
\FXRegisterAuthor{sa}{SA}{\color{red}Sameer}
\FXRegisterAuthor{rt}{RT}{\color{blue}Rekha}
\FXRegisterAuthor{ae}{AE}{\color{orange}Alperen}
\FXRegisterAuthor{ec}{EC}{\color{teal}Erin}

\theoremstyle{plain}
\newtheorem{thm}{Theorem}[section]
\newtheorem{cor}[thm]{Corollary}
\newtheorem{lem}[thm]{Lemma}

\newtheorem{defn}[thm]{Definition}
\newtheorem{prob}[thm]{Problem}

\theoremstyle{definition}

\theoremstyle{remark}
\newtheorem{rem}[thm]{Remark}
\newtheorem{ex}{Example}

\newcommand{\RR}{\mathbb R}
\newcommand{\PP}{\mathbb P}
\newcommand{\CC}{\mathbb C}
\newcommand{\QQ}{\mathbb Q}

\newcommand{\rank}{\operatorname{rank}}

\newcommand{\PGL}{\textup{PGL}}

\title{The Geometry of Rank Drop in a Class of \\Face-Splitting Matrix Products: Part I}

\author{Erin Connelly}
\address{Department of Mathematics, University of Washington, Seattle}
\email{erin96@uw.edu}
\author{Sameer Agarwal}
\address{Google Inc., Seattle}
\email{sameeragarwal@google.com}
\author{Alperen Ergur}
\address{Department of Mathematics, University of Texas at San Antonio}
\email{alperen.ergur@utsa.edu}
\author{Rekha R. Thomas}
\address{Department of Mathematics, University of Washington, Seattle}
\email{rrthomas@uw.edu}

\begin{document}
\maketitle

\begin{abstract}
Given $k$ points 
$(x_i,y_i) \in \PP^2 \times \PP^2$, we characterize rank deficiency of the $k \times 9$ matrix $Z_k$ with rows $x_i^\top \otimes y_i^\top$ in terms of the geometry of the point configurations 
$\{x_i\}$ and $\{y_i\}$.  In this paper we present results for $k \le 6$. For $k \leq 5$, the geometry of the rank-drop locus is characterized by cross-ratios and basic (projective) geometry of point configurations. For the case $k=6$ the rank-drop locus is captured by the classical theory of cubic surfaces. The results for $k=7,8$ and $9$ are presented in the sequel to this paper~\cite{connelly-thomas-vinzant}.
\end{abstract}

\section{Introduction}
Let $\otimes$ denote the Kronecker product~\cite{LOAN200085}. We are interested in solving the following problem:
\begin{prob} \label{prob:main0} 
Given $k$ points $(x_i,y_i)\in\RR^3\times\RR^3$, $k \leq 9$, 
consider the $k \times 9$ 
matrix $Z_k$ whose rows are $x^\top_i \otimes y^\top_i$ for $i=1, \ldots, k$, i.e., 
\[
Z_k = \begin{bmatrix} x^\top_1 \otimes y^\top_1 \\ \vdots \\ x^\top_k \otimes y^\top_k \end{bmatrix}
\]
Delineate the geometry of point configurations  $\{x_i\}$ and $\{y_i\}$ for which $\rank(Z_k) < k$.
\end{prob}

In signal processing, the matrix $Z_k$ is known as the {\em face-splitting product}
~\cite{Slyusar} of the two matrices $X_k \in \RR^{k \times 3}$ with rows $x^\top_1, \ldots, x^\top_k$ and $Y_k \in \RR^{k \times 3}$ with rows $y^\top_1, \ldots, y^\top_k$. The face-splitting product is a special case of the {\em Khatri-Rao matrix product} in linear algebra and statistics~\cite{khatri-rao-1968}. However, our interest~\Cref{prob:main0} comes from an estimation problems in 3D computer vision known as the 7-point problem~\cites{hartley-zisserman-2004}. 
\begin{prob}[7-point problem]
\label{prob:fundamental}
Given 7 points $(\hat{x}_i, \hat{y}_i) \in \RR^2\times \RR^2$ find a matrix $F \in \RR^{3 \times 3}$ such that
\begin{align}
y^\top_i F x_i &= 0,\ \forall i = 1,\dots, 7 ~\text{ and }~\det(F) =0, 
\end{align}
where $x_i = \begin{bmatrix} \hat{x}_i \\ 1 \end{bmatrix}$ and $y_i = \begin{bmatrix} \hat{y}_i \\ 1 \end{bmatrix}$.
\end{prob}

In 3D computer vision the world (or the {\em scene}) is modeled as a collection of points in $\mathbb{P}^3$. Pinhole cameras imaging the scene are modeled as full rank linear projections from  $\mathbb{P}^3$ to $\mathbb{P}^2$.

Let $P_1,\dots, P_7 \in \mathbb{P}^3$ be a scene consisting of seven world points, and $A_1$ and $A_2$ be two pinhole cameras such that they map the scene to images points $x_1,\dots, x_7$ and $y_1,\dots,y_7$ respectively~\footnote{Here we are assuming that the images are finite points in $\mathbb{P}^2$, so they can be de-homogenized to points in $\mathbb{R}^2$}. A solution $F$ to the $7$-point problem that has rank $2$ is known as a {\em fundamental matrix}. The fundamental matrix determines the pair of cameras $A_1$ and $A_2$ up to a homography of $\mathbb{P}^3$~\cites{hartley-zisserman-2004, luong-faugeras}.

Observe that the linear constraints in~\ref{prob:fundamental} can be re-written as~\footnote{Here $\operatorname{vec}(\cdot)$ is the so called {\em vectorization} operator that concatenates the columns of a matrix into a vector.}  
\[(x_i^\top \otimes y_i^\top )\operatorname{vec}(F) = 0,\ \forall i = 1,\dots, k\]
where $\textup{vec}(F)$ is the $9$-dimensional vector obtained by concatenating the columns of $F$.
Thus solving the 7-point problem requires computing the intersection of the null space of a $7\times 9$ face-splitting product with a determinantal variety. Generically it will have three distinct solutions, one of which will always be real. When this is not the case, i.e., Problem~\ref{prob:fundamental} has less than three distinct solutions or has an infinite number of solutions, it is said to be ill-posed~\cite{demmel1997applied}. 

Besides being interesting on its own as a mathematical phenomenon, ill-posedness is of practical interest because of a meta-theorem in numerical analysis that says that the numerical conditioning of a problem is inversely related to the distance to the nearest ill-posed problem~\cites{demmel1987condition, demmel1987geometry, burgisser2013condition}. So, understanding the data for which~\Cref{prob:fundamental} is ill-posed also helps us understand how and when it is numerically hard to solve it using finite precision arithmetic.

There are two ways in which~\Cref{prob:fundamental} can have an infinite number of solutions. The first is that the null space of $Z_7$ has dimension greater than 2. The second is that the null space even though it is 2-dimensional, lies entirely inside the determinantal variety. So, a sufficient condition for ill-posedness is that the corresponding face-splitting product matrix be rank deficient and thus our interest in~\Cref{prob:main0}.

A first answer to~\Cref{prob:main0} is that $\rank(Z_k) < k$ if and only if all the maximal minors of $Z_k$ are zero. These minors are bi-homogeneous polynomials in $x_i$ and $y_i$, and typically do not shed much light on the geometry of  the point configurations $\{x_i \}$ and 
$\{y_i\}$ that cause $Z_k$ to drop rank. 
Our goal is to characterize the rank deficiency of $Z_k$ in terms of the geometry of  $\{x_i\}$ and $\{y_i\}$.

To this end, we begin by observing that the rank of $Z_k$ is a projective invariant.  First, note that multiplying the $x_i$'s and $y_i$'s by non-zero scalars scales the rows of $Z_k$ and does not change its rank.  Second, if $H_1,H_2$ are invertible matrices, then from the mixed-product property of Kronecker products we have that:
\begin{align}
Z'_k = \begin{bmatrix}
    (H_1 x_1)^\top \otimes (H_2 y_1)^\top\\
    \vdots \\
    (H_1 x_k)^\top \otimes (H_2 y_k)^\top
\end{bmatrix} = 
\begin{bmatrix}
    (x_1^\top \otimes y_1^\top)(H^\top_1 \otimes H^\top_2)\\
    \vdots \\
    (x_k^\top \otimes y_k^\top)(H^\top_1 \otimes H^\top_2)
\end{bmatrix} = Z_k (H^\top_1 \otimes H^\top_2).
\end{align}
The matrix $H_1^\top \otimes H_2^\top$ has  full rank since $H_1$ and $H_2$ are invertible, and hence, $\rank(Z'_k) = \rank(Z_k)$. These observations imply that $\rank(Z_k)$ is independent of scaling and choice of coordinates in $\RR^3\times\RR^3$. Therefore, we should be studying the problem over projective space. 
While our computer vision applications care about real input, mathematically there is no loss in assuming that the input points $(x_i,y_i)$ are complex because $\RR\subset\CC$.
Therefore, letting $\PP$ denote projective space over $\CC$ we pass to the following formulation of \Cref{prob:main0} which will be the main subject of 
study in this paper.
\begin{prob} \label{prob:main} 
Given $k$ points $(x_i,y_i)\in\PP^2\times\PP^2$, $k \leq 9$, 
consider the $k \times 9$ 
matrix $Z_k$ whose rows are $x^\top_i \otimes y^\top_i$ for $i=1, \ldots, k$, i.e., 
\[
Z_k= \begin{bmatrix} x^\top_1 \otimes y^\top_1 \\ \vdots \\ x^\top_k \otimes y^\top_k \end{bmatrix}.
\]
Delineate the geometry of point configurations  $\{x_i\}$ and $\{y_i\}$ for which $\rank(Z_k) < k$.
\end{prob}

Let $Z_k\{l\}$ denote a submatrix of $Z_k$ with $l$ of the $k$ rows.  If $\rank(Z_k\{l\}) < l$, then 
$\rank(Z_k) < k$, and any condition on the $l$ points $(x_i,y_i)$ that contribute the rows of  
$Z_k\{l\}$ and cause $Z_k\{l\}$ to drop rank will be a condition on $\{(x_i,y_i),i=1,\ldots, k\}$ such that $Z_k$ drops rank. 
We will call such a condition an {\em inherited condition} for rank drop. For each value of 
$k$ we will be most interested in the {\em non-inherited conditions} on the input that 
make $Z_k$ rank deficient.

As an illustration of the type of answers we will provide, consider the 
 version of  Problem~\ref{prob:main} in $\PP^1 \times \PP^1$. Suppose $\{x_i\}$ and $\{y_i\}$ are each sets of distinct points in $\PP^1$. 
 In 
 Theorem~\ref{thm:P1xP1} we prove that the $4 \times 4$ matrix $Z_4$  
is rank deficient if and only if the 
 cross-ratio of $x_1,x_2,x_3,x_4$ equals that of $y_1,y_2,y_3,y_4$. The cross-ratio is the fundamental projective invariant of four points in 
 $\PP^1$.

 It is a classical fact that the geometry of a finite set of points in $\PP^2$, that remains invariant under the action of $\textup{PGL}(3)$, can be expressed via polynomials in the $3 \times 3$ determinants (brackets) of the matrices whose rows are the points. 
 Many of our answers to Problem~\ref{prob:main} will be phrased in terms of these bracket expressions. The results have analogs in terms of cross-ratios when the points are in general position.

We solve Problem~\ref{prob:main} across this paper (Part I) and its sequel (Part II \cite{connelly-thomas-vinzant}). 
In this paper we address the cases $2 \le k \le 6$ and in the next we address $k=7,8,9$. A great deal of classical algebraic geometry emerges starting with $k=6$. This motivated the split since $k=6$ is rather elaborate and we wanted to keep the papers to 
a reasonable length. Also, the flavor of results in the two papers are slightly different. For $k \leq 6$, we are able to obtain both geometric and algebraic characterizations of rank drop without any assumptions on the points. 
For $k \geq 7$, the tools and results 
become more geometric and requires a type of genericity of the input.

\subsection{Related Work(WIP)}
The study of ill-posedness of Problem~\ref{prob:fundamental} has a long history in computer vision.

One prominent direction is the work on {\em critical configurations}. 
A configuration (or a reconstruction) is a set of cameras and world points in $\mathbb{P}^3$. It is said to be critical if there exist other projectively non-equivalent configurations that produce the same images
~\cites{hartley-zisserman-2004,kahl2002critical,bratelund2024classification,bratelund2views}. There is no simple way to go from these configurations to their images, which is what we need to study the ill-posedness of Problem~\ref{prob:fundamental}.

Perhaps the most well known result on the ill-posedness of~\ref{prob:fundamental} is that if $\{x_i\}$ and $\{y_i\}$ are related by a homography of $\mathbb{P}^2$ then \ref{prob:fundamental} has an infinite number of solutions~\cite{luong-faugeras}. This is far from a complete result as it is only a sufficient condition.

Fan et al.\ study conditioning of~\ref{prob:fundamental}  using a definition based on the world scene \cite{fan2022instability}[Definition 1]. Using this definition allows them to write down a simple formula for the condition number \cite{fan2022instability}[Proposition 2]. Their world-space result~\cite{fan2022instability}[Theorem 2] on the ill-posed locus of~\ref{prob:fundamental} has substantial overlap with the results on critical configurations~\cites{hartley2007critical,  hartley-zisserman-2004, bratelund2views}.  
Their result in image-space~\cite{fan2022instability}[Theorem 4], which is what is relevant to this paper, is algebraic. It is stated in terms of the existence of a certain high degree polynomial that vanishes on $\{x_i\}$ and $\{y_i\}$ as opposed to describing their geometry.

Bertolini, Magri and Turrini have studied critical configurations in multiple settings. The paper closest to our work is \cite{bertolini-magri-turrini} which studies critical 
configurations for $2$-views in terms of $\{x_i\}$ and $\{y_i\}$. 
They show that if the images admit multiple real reconstructions then there is a 
quadratic Cremona transformation taking $x_i \mapsto y_i$. This is also the content of Lemma 3.6 in our paper \cite{connelly-thomas-vinzant}, where we develop this connection further.

\subsection*{Summary of results and organization of the paper.} 

  In Section~\ref{sec:simplifications} we show that one can make some simplifying assumptions on the input to Problem~\ref{prob:main} without any loss of generality. 
  While these assumptions are not needed for the theory, they are helpful and sometimes critical for computations.  
 In Section~\ref{sec:invariant theory tools} we review the tools from invariant theory that are needed in this paper.

In Section~\ref{sec:k less than equal 4} we present our results for $k \leq 4$. 
{\bf Theorem~\ref{thm:k=2,3,4}} solves Problem~\ref{prob:main} for $k=2,3,4$. 
Section~\ref{sec:k=5} considers $k=5$ and presents two theorems.
{\bf Theorem~\ref{thm:k=5}} characterizes the new conditions for rank drop in 
$Z_5$ in terms of bracket equations \eqref{eq: k=5 brackets}. 
A consequence is that when $k=5$, a  necessary condition for rank drop is that one set of points (say the $y_i$'s) is on a line $\ell$. 
The second result is {\bf Theorem~\ref{thm: k=5 geometry theorem}} which gives a geometric characterization of rank deficiency: when the $y_i$'s on the line $\ell$ are distinct and the $x_i$'s are in general position, then the rank drop of $Z_5$ is equivalent to the existence of an isomorphism from the conic $\omega$ through the $x_i$'s to the line $\ell$, taking $x_i$ to $y_i$. 

Finally, we consider the  case $k=6$ in Section~\ref{sec:k=6} where we will see that  Problem~\ref{prob:main} is intimately related with the theory of cubic surfaces in $\PP^3$.  An important new fact here 
is that given $5$ points $(x_i,y_i)$ in general position there is a unique $6$th point $(x_6,y_6)$ that make 
$Z_6$ rank deficient. This $6$th point admits a rational expression in terms of the $5$ points 
({\bf Lemma~\ref{lem:6th point pair}}). 

The results in this section are organized in two parts. In Section~\ref{sec:k=6 algebra} we characterize the rank deficiency of $Z_6$ using brackets; {\bf Theorem~\ref{thm:k=6 neither side in a line}} covers the case of neither the $x_i$'s nor the $y_i$'s being in a line, and {\bf Theorem~\ref{thm: k=6 one side in a line}} addresses the situation in which one set of points is in a line. When the points are in general position, we identify a simple algebraic check ({\bf Theorem~\ref{thm:k=6 general position}}) for rank deficiency in terms of a vector of classical invariants  \eqref{eq:bar vectors are multiples of each other} of $6$ points in $\PP^2$.

In Section~\ref{sec:k=6 geometry}, we explain the origins of the algebraic theorems in Section~\ref{sec:k=6 algebra} 
using geometry. The central result of this section is {\bf Theorem~\ref{thm:k=6 central theorem}} 
which  says that when the points are  in general position, $\rank(Z_6) < 6$ if and only if there is a smooth cubic surface $\mathcal{S}$ that arises as the blow up of $\PP^2$ in both the $x$ points and the $y$ points such that the two sets of exceptional curves form a {\em Schl\"afli double six} on $\mathcal{S}$. This result allows 
for a simple determinantal representation 
of $\mathcal{S}$ using the null space of $Z_6$. 

We then interpret Theorem~\ref{thm:k=6 central theorem} in two different ways. The first is {\bf Theorem~\ref{thm:Werner conics}}
that identifies the unique $6$th point, given $5$ points in general position, that forces $Z_6$ to drop rank. The construction in this theorem relies on a recipe due to Sturm via conics in $\PP^2 \times \PP^2$, and this result 
proves the algebraic Theorem~\ref{thm:k=6 neither side in a line}. 
 
Theorem~\ref{thm:k=6 central theorem} can be rephrased in yet another way using the 
{\em Cremona hexahedral form} of a cubic surface ({\bf Theorem~\ref{thm:canonical blow up} }). This result explains the 
algebraic Theorem~\ref{thm:k=6 general position}. Using the invariants needed in this theorem along with the classical {\em Joubert invariants} of $6$ points in a line, we complete the answer to Problem~\ref{prob:main} when $k=6$ by proving 
Theorem~\ref{thm: k=6 one side in a line}.

\subsection*{A note about computations.} 
All of the algebraic statements in this paper are proven using the computational algebra software package Macaulay2 \cite{M2}. The codes that we used can be found at \\
\centerline{\url{github.com/rekharthomas/rankdrop}}\\
A basic strategy we employ to understand rank drop is to use Macaulay2 to decompose the ideal of maximal minors of $Z_k$ over $\QQ$ into its associated primes (over $\QQ$). Then we interpret these ideals geometrically to state the conditions for rank drop. 
These geometric conditions will hold over $\RR$ and $\CC$ as well since the prime decomposition over $\QQ$ can only refine over the larger fields. 

\subsection*{Acknowledgments} We thank Jarod Alper,  Timothy Duff, and Giorgio Ottaviani for helpful discussions. A.E. is supported by NSF CCF 2110075.

\section{Preprocessing the input}
\label{sec:simplifications}

We begin by showing that one can make simplifying assumptions on the input to Problem~\ref{prob:main} which is a collection of  $k \leq 6$ points $(x_i,y_i) \in \PP^2 \times \PP^2$.  

We  sometimes refer to $(x_i,y_i)$ as a {\em point pair}, especially when we need to decouple the 
components and think of the configurations $x_i$ and $y_i$ separately. It will also be 
convenient to refer to the 
$\PP^2$ containing $\{x_i\}$ as $\PP^2_x$ and the $\PP^2$ containing $\{y_i\}$ as $\PP^2_y$. Let $\textup{PGL}(3)$ be the group of invertible $3 \times 3$ matrices over $\CC$ up to scale. 
An element 
$H \in \textup{PGL}(3)$ induces a homography (projective isomorphism) on $\PP^2$.

In the introduction we informally proved that the rank of $Z_k = (x_i^\top \otimes y_i^\top)_{i=1}^k$ is invariant under  $\textup{PGL}(3)\times \textup{PGL}(3)$ acting on $\PP^2_x \times \PP^2_y$ by left multiplication. This allows us to change bases in $\PP^2_x$ and $\PP^2_y$ independently. We state the result formally below:

\begin{lem} \label{lem:rank drop is invariant under homographies}
Let $(x_i,y_i) \in \PP^2 \times \PP^2$ for $i=1,
\ldots,k$ and let $H_1,H_2 \in \PGL(3)$ be  homographies acting on $\mathbb{P}_x^2$ and $\mathbb{P}^2_y$ respectively (i.e., $(x,y)\mapsto (H_1x,H_2y)$). Then $Z_k = \left (x^\top_i \otimes y^\top_i\right )_{i=1}^k$ has the same rank as $Z'_k = \left((H_1x_i)^\top \otimes( H_2y_i)^\top\right)_{i=1}^k$.
\end{lem}

Lemma~\ref{lem:rank drop is invariant under homographies} allows us to make certain assumptions 
on the configurations $\{x_i\}$ and $\{y_i\}$ which we phrase as corollaries to the lemma. 
While our proofs do not need these assumptions, they become critical for computations in Macaulay2. 
The first assumption is that the $x_i$'s and 
$y_i$'s in the input 
to Problem~\ref{prob:main} are finite points. With a view towards computations, in the following lemma and in several parts of the paper, we fix representatives of points in $\PP^2$ and write them out as vectors in $\CC^3$. 

\begin{cor}\label{cor:assume last coordinate is 1}
Without loss of generality, we can assume that the input to Problem~\ref{prob:main} 
is of the form $x_i = (x_{i1},x_{i2},1)^\top$ and 
$y_i = (y_{i1},y_{i2},1)^\top$ for all $i=1, \ldots, k$, where 
$x_{i1},x_{i2},y_{i1},y_{i2} \in \CC$. Hence  
\begin{align}\label{eq:rowZ}
x_i^\top \otimes y_i^\top = (
x_{i1}y_{i1}, x_{i1}y_{i2} , x_{i1},  x_{i2}y_{i1}, x_{i2}y_{i2},  x_{i2},  y_{i1}, y_{i2}, 1).
\end{align}
\end{cor}

\begin{proof}
Let $x_i \in \PP^2_x$ and $y_i\in\PP^2_y$ be $k$ pairs of points where $k$ is finite, and suppose that some of these points lie on the line at infinity in their $\PP^2$. Since we have only finitely many points, there exists a line $\ell$, common to both $\PP^2$, 
which contains none of these points. Let $H$ be a homography that sends 
$\ell$ to the line at infinity in both $\PP^2$. Then $Hx_i$ and $Hy_i$ are finite points for all $i$. By Lemma~\ref{lem:rank drop is invariant under homographies}, the rank of $Z_k = (x_i^\top \otimes y_i^\top)_{i=1}^k$ equals the rank of  
$Z'_k=((Hx_i)^\top \otimes (Hy_i)^\top)_{i=1}^k$ and hence we can work with the transformed points to understand 
the rank deficiency of $Z_k$.
\end{proof}

The other assumption we can make is that sufficiently generic configurations 
allow some of the points to be fixed, which will prove useful for computations. 

\begin{cor} \label{cor:assumptions}
When no three of the $x_i$'s or $y_i$'s are collinear in their 
respective $\PP^2$, we may 
fix the first four in each set to be:
\begin{align} \label{eq:finite fixing}
\begin{split}
x_1 = y_1 = (0,0,1),\,\,\, x_2=y_2=(1,0,1),\\
x_3=y_3 = (0,1,1), \,\,\,
x_4 = y_4 = (1,1,1).
\end{split}
\end{align}
\end{cor}

At certain times, for the sake of symmetry, we will instead fix the first four points as:
\begin{align} \label{eq:canonical fixing}
\begin{split}
x_1 = y_1 = (1,0,0),\,\,\, x_2=y_2=(0,1,0),\\
x_3=y_3 = (0,0,1), \,\,\,
x_4 = y_4 = (1,1,1).
\end{split}
\end{align}
Since the two fixings are related by a homography, results 
from one fixing transfer to the other. 

\section{Cross-ratios and brackets}
\label{sec:invariant theory tools}

By Lemma~\ref{lem:rank drop is invariant under homographies}, the rank deficiency of $Z_k$ is invariant under the action of $\textup{PGL(3)} \times \textup{PGL}(3)$ on $\PP^2 \times \PP^2$ given by $(H_1,H_2) \cdot (x,y) \mapsto (H_1x,H_2y)$. Therefore, the input to Problem~\ref{prob:main} can be thought of as the 
$\textup{PGL}(3)$-orbits of $k$ points in each $\PP^2$.  
The intrinsic geometry of $k$ points $p_1, \ldots, p_k \in \PP^2$ are the properties of the configuration that remain intact/invariant under the action of $\textup{PGL}(3)$. 
Since an element of $\textup{PGL}(3)$ can be assumed to have determinant $1$, the above geometry  
is precisely the invariant properties, under the action of $\textup{SL}(3)$, of the configuration 
obtained by replacing each $p_i \in \PP^2$ with a representative in $\CC^3$.
We now introduce the basic tools from the invariant theory of point sets needed to formulate our 
results. There are numerous references for this material such as \cite{coble}, \cite{derksen_kemper}, \cite{dolgachev_ortland}, \cite{dolgachev} and \cite{SturmfelsInvariant}.

Let $\mathbb{C}[{\bf p}] := \CC[p_{i1},p_{i2},p_{i3}, i=1, \ldots, k]$ be the polynomial ring in the variables $p_{ij}$ for $i=1,\ldots,k$, $j=1,2,3$, that represent the coordinates of $k$ points 
$p_1, \ldots, p_k$ in $\PP^2$ written via representatives in $\CC^3$. A polynomial $f \in \mathbb{C}[{\bf p}]$ is {\em invariant} under $\textup{SL}(3)$ if $f(p_1, \ldots, p_k) = 
f(Hp_1, \ldots, Hp_k)$ for all $H \in \textup{SL}(3)$. 
It is a well-known classical result that the invariant ring $\mathbb{C}[{\bf p}]^{\textup{SL}(3)}$, which is the collection of all polynomials in $\CC[{\bf p}]$ that are invariant under $\textup{SL}(3)$, is generated by the $3 \times 3$ minors of the matrix whose rows are the symbolic 
$p_1^\top, \ldots, p_k^\top$. In other words, $\mathbb{C}[{\bf p}]^{\textup{SL}(3)}$ is generated as an algebra by the {\em bracket polynomials} 
\begin{equation}
[ijl] :=\det\begin{bmatrix}p_{i1} & p_{i2} & p_{i3}\\p_{j1} & p_{j2} & p_{j3}\\p_{l1} & p_{l2} & p_{l3}\end{bmatrix}
\end{equation}
for all distinct choices of $i,j,l \in \{1, \ldots, k\}$. 
In our situation, we will write $[ijl]_x$ (respectively, $[ijl]_y$) when the 
input points come from $\{x_i\}$ (respectively, $\{y_i\}$). Equations in brackets can be used to express geometry. For example, 
three points $p_i, p_j, p_l \in \PP^2$ are collinear if and only if 
$[ijl]=0$, and $6$ points lie on a conic if and only if they satisfy the bracket equation \eqref{eq:conic condition} in Lemma~\ref{lem: conic formula}.

Often we consider points in a line $\ell$ in $\PP^2$ as points in $\PP^1$ where we can compute the bracket $[ij]$ for $p_i,p_j\in\ell$ by identifying $\ell$ with the line parameterized as say $(*,*,0)$, and then dropping the $0$ coordinate to get
\begin{equation}
[ij] :=\det\begin{bmatrix}p_{i1} & p_{i2}\\p_{j1} & p_{j2}\end{bmatrix}.
\end{equation}

When we have many points $p_1,\ldots,p_k$ on a line and a point $u$ not on that line, the $\PP^2$ brackets $[iju]$ will be closely related to the $\PP^1$ brackets $[ij]$.

\begin{lem}\label{lem:simplifyng 3x3 brackets - line}
For points $p_1,\ldots,p_k$ on a line and a point $u$ not on the line, there exists a nonzero scalar $\lambda$ depending only on the choice of coordinate representative of $u$ such that $[iju]=\lambda[ij]$ for all $i,j$.
\end{lem}
\begin{proof}
We can assume without loss of generality that $p_i=(p_{i1},p_{i2},0)$ for all $i$. We then calculate
\begin{equation}
[iju]=\text{det}\begin{bmatrix}
p_{i1} & p_{i2} & 0\\
p_{j1} & p_{j2} & 0\\
u_1 & u_2 & u_3
\end{bmatrix}=u_3[ij].
\end{equation}
\end{proof}

\begin{defn}
The {\bf cross-ratio} of an ordered list of $4$ points $p_1,\ldots,p_4\in\PP^1$ is
\begin{equation} \label{eq:cross-ratio}
(p_1,p_2;p_3,p_4):=\frac{[13][24]}{[14][23]}.
\end{equation}
\end{defn}

The cross-ratio is the only projective invariant of 
$4$ points in $\PP^1$ in the following sense. 

\begin{lem} \label{lem:cr=homography}
If $p_1,\ldots,p_4\in\PP^1$ and $q_1,\ldots,q_4\in\PP^1$ are two sets of points then $(p_1,p_2;p_3,p_4)=(q_1,q_2;q_3,q_4)$ if and only if there exists a homography $H:\PP^1\to\PP^1$ such that $Hp_i=q_i$ for all $i$. 
\end{lem}

Permuting the points on the line changes the cross-ratio in a systematic way, i.e., if 
\begin{equation}
(p_1,p_2;p_3,p_4)=(q_1,q_2;q_3,q_4)
\end{equation}
then
\begin{equation}
(p_{\sigma(1)},p_{\sigma(2)};p_{\sigma(3)},p_{\sigma(4)})=(q_{\sigma(1)},q_{\sigma(2)};q_{\sigma(3)},q_{\sigma(4)})
\end{equation}
for all $\sigma\in S_4$.
Thus we  do not need to consider multiple orderings.
The cross-ratio is $0,1,\infty$, or $\frac{0}{0}$ if and only if the points are not distinct. See \cite{semple-kneebone}*{III, \S 4-5}.

\begin{ex}\label{ex:brackets better} We now give an example to illustrate that equations in 
brackets can be more robust than equalities of cross-ratios.
Consider the following points in $\PP^1$ where the first set has repetition:
\begin{align*}
x_1&=(1,1)\quad x_2=(1,1) &\quad\quad
y_1&=(8,1)\quad y_2=(4,1)\\
x_3&=(1,1)\quad x_4=(3,1) &\quad\quad
y_3&=(2,1)\quad y_4=(5,1).
\end{align*}
Then 
\begin{equation}
(x_1,x_2;x_3,x_4)=\frac{0}{0},\quad\quad (y_1,y_2;y_3,y_4)=-1
\end{equation}
and therefore the cross-ratios are not equal. 
However, if we write the cross-ratio equality 
\begin{equation}\label{eq:cross-ratios equality}
(x_1,x_2;x_3,x_4)= (y_1,y_2;y_3,y_4)
\end{equation}
in bracket form as 
\begin{equation}\label{eq:bracket equality}
[13]_x[24]_x[14]_y[23]_y=[13]_y[24]_y[14]_x[23]_x
\end{equation}
then equality does hold. 
\end{ex}

In Example~\ref{ex:brackets better}, the points $(x_i,y_i)$ lie on the variety defined by the bracket equation \eqref{eq:bracket equality} but do not satisfy the cross ratio equality \eqref{eq:cross-ratios equality}. 
The discrepancy is because cross-ratios are well-defined only on an open set, while the corresponding bracket equation holds on the Zariski closure of this open set. Since our interest is in describing the algebraic variety of input points that make $Z_4$ rank deficient, the bracket expressions are the correct choice. 
However, whenever the points are in sufficiently general position, we can pass to cross-ratio expressions which, besides being elegant, have the advantage that they can be visualized easily. 

We will also need planar and conic cross-ratios which we now define.
\begin{defn}
The {\bf $5$ point cross-ratio} of $p_1,\ldots,p_5\in\PP^2$ is 
\begin{equation} \label{eq:5 point cross-ratio}
 (p_1,p_2;p_3,p_4;p_5) :=\frac{[135][245]}{[145][235]}.
\end{equation}
This is also called a {\bf planar cross-ratio} or the {\bf cross-ratio around} $p_5$. 
\end{defn}

Note that planar cross-ratios are preserved under a homography of $\PP^2$.  For $5$ distinct points, the cross-ratio around $x_5$ can be obtained geometrically by drawing the $4$  lines $\overline{p_ip_5}$ for $i=1,2,3,4$, cutting them with a transversal, and then computing the cross-ratio of the intersection points. See Figure~\ref{fig:planar cross-ratio}. 
Also, the planar cross-ratio is transformed by permutations of $p_1, \ldots,p_4$ similarly to the line cross-ratio.

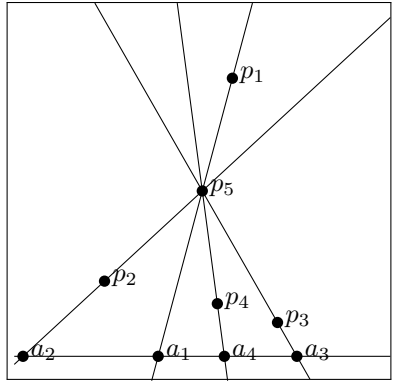
\begin{figure}
\fbox{
\begin{picture}(5,5)
\put(2.5,2.5){{\circle*{0.15}}\hbox{\kern3pt \texttt{$p_5$}}}
\put(2.9,4){{\circle*{0.15}}\hbox{\kern3pt \texttt{$p_1$}}}
\put(1.2,1.3){{\circle*{0.15}}\hbox{\kern3pt \texttt{$p_2$}}}
\put(3.5,0.75){{\circle*{0.15}}\hbox{\kern3pt \texttt{$p_3$}}}
\put(2.7,1){{\circle*{0.15}}\hbox{\kern3pt \texttt{$p_4$}}}
\put(0.11667,0.3){{\circle*{0.15}}\hbox{\kern3pt \texttt{$a_2$}}}
\put(1.91333,0.3){{\circle*{0.15}}\hbox{\kern3pt \texttt{$a_1$}}}
\put(2.79333,0.3){{\circle*{0.15}}\hbox{\kern3pt \texttt{$a_4$}}}
\put(3.7571,0.3){{\circle*{0.15}}\hbox{\kern3pt \texttt{$a_3$}}}
\thinlines
\put(0,0.3){\line(1,0){5}}
\put(3.1667,5){\line(-0.4,-1.5){1.34}}
\put(0,0.19231){\line(1,0.92307692){5}}
\put(1.071429,5){\line(1,-1.75){2.86}}
\put(2.166667,5){\line(1,-7.5){0.67}}
\end{picture}}
\caption{For distinct points $p_1,\ldots,p_5 \in \PP^2$, the planar cross-ratio $(p_1,p_2;p_3,p_4;p_5)$ equals the $4$ point cross-ratio $(a_1,a_2;a_3,a_4)$ \label{fig:planar cross-ratio}}.
\end{figure}

\begin{lem}\label{lem: conic formula}\cite{SturmfelsInvariant}*{Example 3.4.3}
A collection of $6$ points $p_1,\ldots,p_6 \in \PP^2$ lie on a conic if and only if 

\begin{equation} \label{eq:conic condition}
[135][245][146][236]=[136][246][145][235], 
\end{equation}
or generically, the following cross-ratio equality holds: 
\begin{equation}
(p_1,p_2;p_3,p_4;p_5)=(p_1,p_2;p_3,p_4;p_6).
\end{equation}
\end{lem}

Lemma~\ref{lem: conic formula} makes the following definition of a conic cross-ratio well-defined. 

\begin{defn}\label{def:conic cross-ratio}
The {\bf conic cross-ratio}
of $4$ points $p_1,\ldots,p_4$ on a (non-degenerate) conic $\omega$  is
\begin{equation}
(p_1,p_2;p_3,p_4)_\omega:=(p_1,p_2;p_3,p_4;p)
\end{equation}
where $p\in \omega$ is any new point.
\end{defn}

Any non-degenerate conic $\omega$ is isomorphic to $\PP^1$. If $\iota:\omega\to\PP^1$ is one such isomorphism then
\begin{equation}
(p_1,p_2;p_3,p_4)_\omega=(\iota(p_1),\iota(p_2);\iota(p_3),\iota(p_4)).
\end{equation}

\section{$k \leq 4$}
\label{sec:k less than equal 4}
We will now answer Problem~\ref{prob:main} in the cases $k \leq 4$. As a warm up to Problem~\ref{prob:main} we first consider the analogous question in $\PP^1 \times \PP^1$.
Recall that any condition on a proper subset of $\{(x_i,y_i)\}$ that make the 
corresponding rows of $Z_k$ dependent is called an inherited condition for the rank drop 
of $Z_k$. 

\begin{thm} \label{thm:P1xP1}
Consider the $k \times 4$ matrix $Z_k = (x_i \otimes y_i)_{i=1}^k$ where $(x_i,y_i) \in \PP^1\times\PP^1$ and $k \leq 4$. Then, 
\begin{enumerate}
\item $\rank(Z_2) < 2$ if and only if $x_1=x_2$ and $y_1=y_2$,
\item $\rank(Z_3) < 3$ if and only if either an inherited condition holds or $x_1 = x_2 = x_3$ or $y_1 = y_2 = y_3$, 
\item $\rank(Z_4) < 4$ if and only if the following bracket equation holds: 
    \begin{equation}\label{eq:detZ}
[13]_x[24]_x[14]_y[23]_y=[14]_x[23]_x[13]_y[24]_y.
    \end{equation}
Generically, $x_1, \ldots, x_4$ and $y_1, \ldots, y_4$ are distinct and then \eqref{eq:detZ} holds if and only if there is a homography $H\in \textup{PGL}(2)$ such that $Hx_i=y_i$ for $i=1,\ldots,4$.
\end{enumerate}
\end{thm}

\begin{proof} In each of the above cases, $Z_k$ will be 
rank deficient if and only if all its maximal minors vanish.
Hence we compute the ideal of $k \times k$ minors of the symbolic $Z_k$ in each case and decompose it to understand rank deficiency. While computations in Macaulay2 are done over $\QQ$, our results hold over $\CC$ as explained at the end of the Introduction.
By Corollary~\ref{cor:assume last coordinate is 1} we can fix the last coordinates of $x_i$ and $y_i$ to $1$.

If $k=2$ the ideal of $2 \times 2$ minors of $Z_2$ is the prime ideal 
$\langle x_{11}-x_{21}, y_{11}-y_{21}\rangle$
which says that rank drops if and only if 
$x_1 = x_2$ and $y_1 = y_2$. We will call this the {\em repeated point condition}.

If $k=3$, the ideal of $3 \times 3$ minors of the symbolic $Z_3$ matrix is radical and decomposes into $5$ prime ideals. 
Three of them correspond to repeated points 
$x_i = x_j$ and $y_i=y_j$ for $i \neq j$. These are inherited for rank drop from the $k=2$ case. 
The remaining $2$ prime ideals say that $x_1 = x_2 = x_3$ and 
$y_1 = y_2 = y_3$ respectively.

When $k=4$, we can use Macaulay2 to check that
\begin{equation*}
\det(Z_4)=[14]_x[23]_x[13]_y[24]_y-[13]_x[24]_x[14]_y[23]_y 
\end{equation*}
which yields \eqref{eq:detZ}. 
In this case, the ideal of maximal minors is principal and prime, generated by $\det(Z_4)$, and does not decompose. 
However, note that the bracket expression for $\det(Z_4)$ vanishes under all of the conditions inherited from $k=2$ and $k=3$. Furthermore, when the $x_i$'s and $y_i$'s are distinct then \eqref{eq:detZ} is equivalent to $(x_1,x_2;x_3,x_4)=(y_1,y_2;y_3,y_4)$ and the homography follows by Lemma \ref{lem:cr=homography}.
\end{proof}

We now consider Problem~\ref{prob:main} for $k \leq 4$ points $(x_i, y_i) \in \PP^2 \times \PP^2$. Again all of the statements can be checked in Macaulay2 by following the same general strategy as in the proof of Theorem~\ref{thm:P1xP1}. 

\begin{thm} \label{thm:k=2,3,4}
Suppose $(x_i,y_i) \in \PP^2 \times \PP^2$ and 
$Z_k = (x_i^\top \otimes y_i^\top)_{i=1}^k$ for $k \leq 4$. Then
\begin{enumerate}
\item $\rank(Z_2) < 2$ if and only if $x_1=x_2$ and $y_1 = y_2$.
\item $\rank(Z_3) <3$ if and only if either an inherited condition holds or 
 $x_1=x_2=x_3$ 
 and  $y_1, y_2,y_3$ are on a line, or  
$x_1,x_2,x_3$ are on a line and 
    $y_1 = y_2 = y_3$. 
\item $\rank(Z_4) < 4$ if and only if 
an inherited condition holds or 
one of the following is true:
\begin{enumerate}
\item $x_1 =x_2=x_3=x_4$ or $y_1 = y_2 = y_3 = y_4$.
\item The points $x_1, \ldots, x_4 \in \PP^2_x$ are in a line and 
$y_1, \ldots, y_4 \in \PP^2_y$ are in a line,  and
\begin{equation} \label{eq:k=4 bracket equation}
[13]_x[24]_x[14]_y[23]_y=[14]_x[23]_x[13]_y[24]_y.
\end{equation}
Generically, the collinear points $x_1, \ldots, x_4$ and $y_1, \ldots, y_4$ are distinct and then  \eqref{eq:k=4 bracket equation} holds if and only if there is a homography $H\in \textup{PGL}(3)$ such that $Hx_i=y_i$ for $i=1,\ldots,4$.
\end{enumerate}    
    \end{enumerate}
\end{thm}

\begin{proof}
The proof of (1) and (2) can be done using Macaulay2 by computing the ideal of maximal minors of $Z_k$ and then decomposing the radical of the ideal to get all possible conditions for rank drop. The matrix $Z_k$ is written symbolically as in \eqref{eq:rowZ} where, by Corollary \ref{cor:assume last coordinate is 1}  we assume that the last coordinate of all $x_i$'s and $y_i$'s are $1$. 
For instance, the ideal of 
$3 \times 3$ minors of $Z_3$ is radical and of dimension $8$ and degree $3$. It is the intersection of $5$ prime ideals of which $3$ correspond to the repeated point conditions inherited from $k=2$. The fourth ideal provides the new condition that $x_1 = x_2 = x_3$ and $y_1,y_2,y_3$ are in a line. The fifth ideal provides the analogous condition with the roles of $x$ and $y$ switched. These ideals appear in the form:
$$ \langle x_{11}-x_{21}, x_{12}-x_{22}, x_{11}-x_{31}, x_{12}-x_{32}, [123]_y \rangle \quad \textup{ and } \quad \langle y_{11}-y_{21}, y_{12}-y_{22}, y_{11}-y_{31}, y_{12}-y_{32}, [123]_x \rangle.$$

The ideal of maximal minors of $Z_4$ is radical of dimension $12$ and degree $6$. 
It decomposes into $17$ primes ideals:
\begin{itemize}
\item $6$ of the primes correspond to repeated points (inherited from $k=2$). All of these have dimension $12$ 
and degree $1$.
\item $8$ of the primes correspond to $3$ points in one side coinciding and the corresponding points on the other side on a line (inherited from $k=3$). All of these ideals 
have dimension $11$ and degree $2$.
\item $2$ of the primes say $x_1=x_2=x_3=x_4$ and 
$y_1 = y_2 = y_3 = y_4$ respectively (new conditions for $k=4$). These ideals have dimension $10$ and degree $1$.
\item The last prime component, call it $J$, has dimension $11$ and degree $24$. It implies that the $x_i$ and $y_i$ are in a line (new condition for $k=4$). One way to verify this claim is to 
check that the ideal generated by the brackets $[123]_x, [124]_x$ and $[123]_y,[124]_y$ is contained in $J$. These brackets being $0$ encode the collinearity of $x_i$'s and the collinearity of $y_i$'s. 
\end{itemize}
One can understand $J$ more precisely by making the following calculation. If the $x$ and $y$ points are both in a line, then we can use homographies to fix these lines, and 
assume that  $x_i=(a_i,0,1)$ and $y_i=(b_i,0,1)$ for $i=1,\ldots,4$. This means that we can identify 
$x_i$ with $(a_i,1) \in \PP^1$ and $y_i$ with 
$(b_i,1) \in \PP^1$. Under this substitution, $Z_4$ has exactly one non-zero  minor
\begin{equation} \label{eq:det M} 
    \det(M)=\det\begin{bmatrix}
    a_1b_1 & a_1 & b_1 & 1\\
    a_2b_2 & a_2 & b_2 & 1\\
    a_3b_3 & a_3 & b_3 & 1\\
    a_4b_4 & a_4 & b_4 & 1\end{bmatrix} = [14]_x[23]_x[13]_y[24]_y-[13]_x[24]_x[14]_y[23]_y.
\end{equation}
 The second equality in 
\eqref{eq:det M} can be checked in Macaulay2.  The bracket expression set to $0$ is the equality of cross-ratios: 
\begin{align*}
(x_1,x_2;x_3,x_4) = (y_1,y_2;y_3,y_4)
\end{align*}
 which becomes well-defined when the $x_i$'s are distinct and the $y_i$'s are distinct. 
\end{proof}

\section{$k=5$} \label{sec:k=5}
We now consider the rank deficiency of $Z_5$. In this case, we present two theorems. \Cref{thm:k=5} is a characterization of rank deficiency in terms of brackets. 
We will see that a necessary condition for rank drop is that one set of points (say $\{y_i\}$) is on a line $\ell$. 
\Cref{thm: k=5 geometry theorem} gives a geometric characterization of rank deficiency when the points are sufficiently generic, i.e.,  when the $y_i$'s on the line $\ell$ are distinct and the $x_i$'s are in general position. In this case the rank drop of $Z_5$ is equivalent to the existence of an isomorphism from the conic $\omega$ through the $x_i$'s to the line $\ell$ taking $x_i$ to $y_i$. This is analogous to Theorem~\ref{thm:k=2,3,4} (3)(b) where there is a 
homography taking $x_i$ to $y_i$ under sufficient genericity.
We will illustrate both results using Example~\ref{ex:k=5}.

\begin{thm} \label{thm:k=5}
Given a configuration $\{(x_i,y_i)\}_{i=1}^5$ in $\PP^2 \times \PP^2$, the matrix $Z_5$ is rank deficient if and only if either, one of the inherited conditions hold, or the points in one $\PP^2$ (say in $\PP^2_y$) are on a line and the following bracket equations hold for all distinct $i_1,i_2,i_3,i_4,j\in\{1,2,3,4,5\}$:
\begin{equation}\label{eq: k=5 brackets}
[i_1i_3j]_x[i_2i_4j]_x[i_1i_4]_y[i_2i_3]_y
=[i_1i_4j]_x[i_2i_3j]_x[i_1i_3]_y[i_2i_4]_y.
\end{equation}
By the action of permutations, there are only $5$ equations to consider in 
\eqref{eq: k=5 brackets}, one for each choice of $j\in\{1,2,3,4,5\}$. In the case that both the $x_i$'s and the $y_i$'s are contained in lines then \eqref{eq: k=5 brackets} holds automatically.
Generically, \eqref{eq: k=5 brackets} is equivalent to 
the cross-ratio equalities: 
\begin{equation}\label{eq: k=5 cross-ratio equalities}
(x_{i_1},x_{i_2};x_{i_3},x_{i_4};x_j)
=(y_{i_1},y_{i_2};y_{i_3},y_{i_4})\end{equation}
for all distinct $i_1,i_2,i_3,i_4,j\in\{1,2,3,4,5\}$.
\end{thm}

\begin{proof}
The proof splits into 
two cases. For the first case, we assume that the points in neither side are entirely in a line and show that rank will drop only under inherited conditions. For the second case, we are then free to assume that one side is contained in a line, which allows us to simplify the Macaulay2 calculations.
\begin{enumerate}
    \item Suppose neither the $x_i$'s nor the $y_i$'s are in a line, and none of the inherited conditions hold. Then we can assume without loss of generality that $x_1 \neq x_2$ and $y_1 \neq y_2$ because no side has four coincident points (due to the absence of inherited conditions).
    Since neither side is contained entirely in a line, we can also assume that $x_1,x_2,x_3$ are not collinear. 
    Then either $y_1,y_2,y_3$ are also not collinear, or $y_1,y_2,y_3$ are in a line  which means that we can assume $y_1,y_2,y_4$ are 
    not collinear. Applying homographies, we therefore need to work with two different 
    fixings of points:
\begin{equation}\label{eq:k=5 fix1}
    x_1=(1,0,0)=y_1\quad\quad
    \textup{ and } \quad\quad x_2=(0,1,0)=y_2\quad\quad \textup{ and } \quad\quad x_3=(0,0,1)=y_3
\end{equation} 
and 
\begin{equation}\label{eq:k=5 fix2}
    x_1=(1,0,0)=y_1\quad\quad
    \textup{ and } \quad\quad x_2=(0,1,0)=y_2\quad\quad \textup{ and } \quad\quad x_3=(0,0,1)=y_4.
\end{equation} 
Note that these fixings are not to finite points as we have done so far, but are equivalent 
to the fixing in \eqref{eq:finite fixing} as remarked after Corollary~\ref{cor:assume last coordinate is 1}. They allow for faster computations in Macaulay2.

In \eqref{eq:k=5 fix1}, the ideal generated by the $5\times 5$ minors of $Z_5$ is the intersection of $20$ prime ideals. Of these, $4$ correspond to invalid inputs $x_i,y_i=(0,0,0)$; $7$ correspond to the inherited repeated point condition from Theorem~\ref{thm:k=2,3,4}(1); $6$ correspond to the inherited condition from Theorem~\ref{thm:k=2,3,4}(2); and the remaining $3$ correspond to the inherited bracket/cross-ratio condition from Theorem~\ref{thm:k=2,3,4}(3)(b). 

In \eqref{eq:k=5 fix2}, the ideal generated by the $5\times5$ minors of $Z_5$ is the intersection of $17$ prime ideals. Of these, $4$ correspond to invalid inputs $x_i,y_i=(0,0,0)$; $7$ correspond to repeated points; $4$ correspond to the inherited condition from Theorem~\ref{thm:k=2,3,4}(2); and the remaining $2$ correspond to the inherited bracket/cross-ratio condition from Theorem~\ref{thm:k=2,3,4}(3)(b). 

Thus, there are no new conditions for rank drop if neither the $x_i$'s nor the $y_i$'s are 
in a line. 

\item Suppose $y_1, \ldots, y_5$ are on a line. Applying
homographies, we may assume that  $y_i=(m_i,0,1)$ for $i=1,\ldots,5$ and all $x_i$ are finite. We may further assume that $x_1=(0,0,1)$ and that $y_1=(0,0,1)$. Then rearranging the columns of $Z_5$ we 
get 
\begin{equation}
   Z'_5 =  \begin{bmatrix} 
    0 & 0 & 0 & 0 & 0 & 1 & 0 & 0 & 0\\
    x_{21}m_2 & x_{22}m_2 & x_{21} & x_{22} & m_2 & 1 & 0 & 0 & 0\\
    x_{31}m_3 & x_{32}m_3 & x_{31} & x_{32} & m_3 & 1 & 0 & 0 & 0\\
    x_{41}m_4 & x_{42}m_4 & x_{41} & x_{42} & m_4 & 1 & 0 & 0 & 0\\
    x_{51}m_5 & x_{52}m_5 & x_{51} & x_{52} & m_5 & 1 & 0 & 0 & 0
    \end{bmatrix}.
\end{equation}
Let $I$ be the ideal generated by the $6$ non-zero $5 \times 5$ minors of $Z'_5$ and let $J$ be the ideal generated by the $5$ differences of left and right hand sides in the bracket equations \eqref{eq: k=5 brackets} under the same fixings of points. Then 
we can verify using Macaulay2 that $I=\sqrt{J}$. A symmetric argument will switch the roles of the $x$ and $y$ points. 
\end{enumerate}
\end{proof}

\begin{ex}\label{ex:k=5}
The following example illustrates Theorem~\ref{thm:k=5} and motivates Theorem $\ref{thm: k=5 geometry theorem}$.
Consider 
\begin{align*}
x_1=(0,0,1)\quad& y_1=(0,0,1)  \quad\quad&x_4=(1,1,1)\quad&y_4=(-4,0,1)\\ 
x_2=(1,0,1)\quad&y_2=(1,0,1) \quad\quad&x_5=(50,98,113)\quad& y_5=(8,0,1)\\
x_3=(0,1,1)\quad& y_3=(3,0,1)
\end{align*}
Here the $y_i$'s are on the line $y_i=(m_i,0,1)$ and we can verify that \eqref{eq: k=5 brackets} is satisfied. The matrix $Z_5$ shown below, has rank $4$.
$$Z_5=\begin{bmatrix}
0 & 0 & 0 & 0 & 0 & 0 & 0 & 0 & 1\\
1 & 0 & 1 & 0 & 0 & 0 & 1 & 0 & 1\\
0 & 0 & 0 & 3 & 0 & 1 & 3 & 0 & 1\\
-4 & 0 & 1 & -4 & 0 & 1 & -4 & 0 & 1\\
400 & 0 & 50 & 784 & 0 & 98 & 904 & 0 & 113
\end{bmatrix}.$$

Note that in this example, the $y_i$'s (which are on a line) are distinct, and that no $3$ of the 
$x_i$'s are in a line. 
In this case, there is a $1$-parameter family of rank $2$ projective transformations $\{T_\lambda\}$ such that $T_\lambda (x_i)=y_i$ for $i=1,\ldots,5$ except for $5$ unique values of $\lambda$, one for each of $x_1, \ldots, x_5$. In this example, this is the family 
\begin{equation}
T_\lambda=\begin{bmatrix}
0 & 12 & 0\\
-7 & -3 & 7
\end{bmatrix}+\lambda\begin{bmatrix}
7 & -15 & 0\\
7 & -5 & 0
\end{bmatrix}.
\end{equation}
For each $i=1,\ldots,5$ there exists unique $\lambda_i$ such that $T_{\lambda_i}$ has nullvector (center) $x_i$. 
These are
\begin{equation}
\lambda_1=\infty\quad \lambda_2=0\quad\lambda_3=\frac{4}{5}\quad\lambda_4=\frac{3}{2}\quad\lambda_5=\frac{21}{20}.
\end{equation}
Since $T_{\lambda_i} (x_i) =(0,0,0)\not\sim y_i$, these $5$ transformations can be seen as degenerate members of the family.
\end{ex}

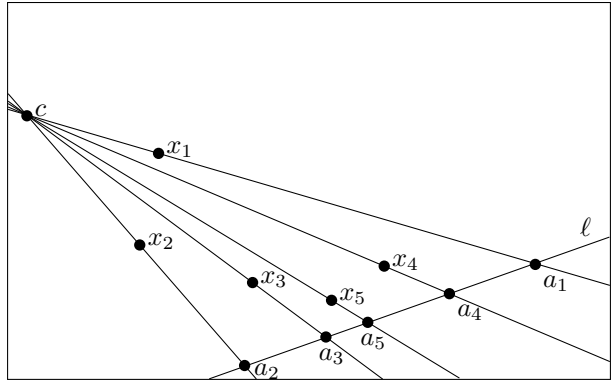
\begin{figure}
\fbox{%
\begin{picture}(8,5)
\put(2,3){{\circle*{0.15}}\hbox{\kern3pt \texttt{$x_1$}}}
\put(1.75,1.78125){{\circle*{0.15}}\hbox{\kern3pt \texttt{$x_2$}}}
\put(3.25,1.28125){{\circle*{0.15}}\hbox{\kern3pt \texttt{$x_3$}}}
\put(5,1.5){{\circle*{0.15}}\hbox{\kern3pt \texttt{$x_4$}}}
\put(4.3,1.045){{\circle*{0.15}}\hbox{\kern3pt \texttt{$x_5$}}}

\put(0.25,3.5){{\circle*{0.15}}\hbox{\kern3pt \texttt{$c$}}}

\put(7.0088,1.52278){{\circle*{0.15}}}
\put(7.0088,1.22278){\hbox{\kern3pt \texttt{$a_1$}}}
\put(3.144,0.17768){{\circle*{0.15}}}
\put(3.164,0.05){\hbox{\kern3pt \texttt{$a_2$}}}
\put(4.225,0.557){{\circle*{0.15}}}
\put(4.025,0.257){\hbox{\kern3pt \texttt{$a_3$}}}
\put(5.8666,1.133){{\circle*{0.15}}}
\put(5.8666,0.833){\hbox{\kern3pt \texttt{$a_4$}}}
\put(4.7811,0.7521){{\circle*{0.15}}}   
\put(4.5811,0.4521){\hbox{\kern3pt \texttt{$a_5$}}}

\thinlines
\put(2.67376,0){\line(4.3264,1.534){5.32}}

\put(0,3.5854){\line(1,-0.29273){8}}
\put(0,3.7943){\line(1,-1.1503){3.3}}
\put(0,3.6624){\line(1,-0.6087){6}}
\put(0,3.69504){\line(1,-0.7427){4.98}}
\put(0,3.6173){\line(1,-0.42346){8}}

\put(7.5,1.9){\hbox{\kern3pt \texttt{$\ell$}}}
\end{picture}}
\caption{We construct the projection with center $c$.\label{fig: k=5 projection construction}}
\end{figure}

We now formalize the second half of Example~\ref{ex:k=5} as a theorem. As seen in Theorem~\ref{thm:k=5}, for $Z_5$ to drop rank, either the $x_i$'s or the $y_i$'s have to be in a line. 
Suppose the $y_i$'s are distinct and in a line, and the $x_i$'s are in general position.
Then the $x_i$'s lie on a conic $\omega$ and we will see that there is an 
isomorphism taking $\omega$ to the line containing the $y_i$'s so that each $x_i$ maps to $y_i$. 

\begin{thm}\label{thm: k=5 geometry theorem}
Suppose we have a configuration $\{(x_i,y_i)\}_{i=1}^5$ such that no $3$ of the $x_i$'s are collinear and the $y_i$'s are distinct and on a line $\ell_y$.  Let $\omega$ be the conic through $x_1,\ldots, x_5$. Suppose further that no inherited condition for rank drop holds. Then the following are equivalent: 
\begin{enumerate}
    \item The matrix $Z_5$ is rank deficient.
    \item For all $c\in\omega \setminus \{x_1,\ldots,x_5\}$ there exists a projective transformation $T_c \,:\, \PP^2 \rightarrow \PP^1$ centered at $c$ such that $T_c (x_i)\sim y_i$ for all $i$.
    \item For some $c\in\omega \setminus \{x_1,\ldots,x_5\}$ there exists a projective transformation $T_c \,:\, \PP^2 \rightarrow \PP^1$ centered at $c$ such that $T_c(x_i)\sim y_i$ for all $i$.
    \item There exists a cross-ratio preserving isomorphism $F:\omega\to\ell_y$ such that $F(x_i)=y_i$ for all $i$.
\end{enumerate} 
\end{thm}

\begin{proof}
To prove (1) implies (2), let $c\in\omega \setminus \{x_1, \ldots, x_5\}$. Then for all distinct $i_1,i_2,i_3,i_4,j\in\{1,\ldots,5\}$  
\begin{equation}
(y_{i_1},y_{i_2};y_{i_3},y_{i_4})=(x_{i_1},x_{i_2};x_{i_3},x_{i_4};x_j)=(x_{i_1},x_{i_2};x_{i_3},x_{i_4})_\omega=(x_{i_1},x_{i_2};x_{i_3},x_{i_4};c)
\end{equation}
for  any $c \in \omega$. The first equation is from \eqref{eq: k=5 cross-ratio equalities} in Theorem~\ref{thm:k=5} and the other two equalities follow by Definition \ref{def:conic cross-ratio}.
Let $\overline{cx_i}$ be the line through $c$ and $x_i$ and let $\ell$ be any line such that $c\not\in\ell$. Then if we define $a_i$ to be the unique intersection of $\ell$ and $\overline{cx_i}$ it follows that
\begin{equation}
(a_{i_1},a_{i_2};a_{i_3},a_{i_4})=(x_{i_1},x_{i_2};x_{i_3},x_{i_4};c)=(y_{i_1},y_{i_2};y_{i_3},y_{i_4})
\end{equation}
for all $i_1,i_2,i_3,i_4\in\{1,\ldots,5\}$. See Figure \ref{fig: k=5 projection construction}. It follows that there exists homography $H$ such that $Ha_i\sim y_i$ for all $i$. Therefore if $T'_c$ is the projection in $\PP^2_x$, centered at $c$, onto the line $\ell$ then $T_c :=HT'_c$ is the desired projective transformation with $T_c(x_i)\sim y_i$ for all $i$.

Clearly (2) implies (3). To show that (3) implies (1), suppose such a projection exists and is centered at $c$. Then
\begin{equation}
(x_{i_1},x_{i_2};x_{i_3},x_{i_4};x_j)=(x_{i_1},x_{i_2};x_{i_3},x_{i_4})_\omega=(x_{i_1},x_{i_2};x_{i_3},x_{i_4};c)=(y_{i_1},y_{i_2};y_{i_3},y_{i_4})
\end{equation}
for all distinct $i_1,i_2,i_3,i_4,j\in\{1,\ldots,5\}$. The first two equalities follow by Definition \ref{def:conic cross-ratio} and the last equality follows by hypothesis. It immediately follows that $Z_5$ is rank deficient by Theorem~\ref{thm:k=5}.

To show that (1) implies (4) we may assume that the cross-ratio equalities \eqref{eq: k=5 cross-ratio equalities} hold.  
Then define an isomorphism $F:\omega\to\ell_y$ by $F(x_1)=y_1, F(x_2)=y_2, F(x_3)=y_3$. Then for all $p\in\omega \setminus \{x_1,x_2,x_3\}$, $F(p)=q$ is such that $(x_1,x_2;x_3,p)_\omega=(y_1,y_2;y_3,q)$. It then follows that $F(x_4)=y_4$ and $F(x_5)=y_5$ because
\begin{equation}
(x_1,x_2;x_3,x_4)_\omega=(y_1,y_2;y_3,y_4)\quad\quad(x_1,x_2;x_3,x_5)_\omega=(y_1,y_2,y_3;y_5)
\end{equation}
by hypothesis (Theorem~\ref{thm:k=5}).

Finally, to show that (4) implies (1), suppose such an isomorphism $F$ exists. Then
\begin{equation}
(x_{i_1},x_{i_2};x_{i_3},x_{i_4};x_j)=(x_{i_1},x_{i_2};x_{i_3},x_{i_4})_\omega=(F(x_{i_1}),F(x_{i_2});F(x_{i_3}),F(x_{i_4}))=(y_{i_1},y_{i_2};y_{i_3},y_{i_4})
\end{equation}
for all $i_1,i_2,i_3,i_4,j\in\{1,\ldots,5\}$. The first equality follows by Definition \ref{def:conic cross-ratio} and the second and third equalities follow by hypothesis. It then follows that $Z_5$ is rank deficient by Theorem~\ref{thm:k=5}.
\end{proof}


\section{$k=6$}
\label{sec:k=6}

The final case we will consider in this paper is a set of $6$ inputs $\{(x_i,y_i)\}_{i=1}^6 \in \PP^2 \times \PP^2$. 
We organize our results in two parts. In the first part we state algebraic characterizations of the rank deficiency of $Z_6$. These are analogous to the results in the previous sections in that the results are in terms of invariant theory. 
In the second part, we explain the geometry that leads to these algebraic statements. This relies on the classical theory of cubic surfaces and its combinatorics. The geometry 
works in an open region where the input is sufficiently generic.   
However, as we saw in previous sections, the algebra that it translates to 
describes the Zariski closure of these open regions and hence the set of all inputs that make $Z_6$ rank deficient.

We say that $p_1, \ldots, p_6 \in \PP^2$ are in {\bf general position} if 
at most $2$ points are in a line and at most $5$ points are on a conic. 
In several results below we will need that the $x$ points 
and/or the $y$ points are in general position. We will also need that a collection of point pairs $\{(x_i,y_i)\} \subset \PP^2 \times \PP^2$ is in \textbf{general position}, by which we will mean that
\begin{enumerate}
\item the $x$ points and $y$ points are in general position, and 
\item no more than $4$ pairs $(x_i,y_i)$ are related by a homography $Hx_i=y_i$.
\end{enumerate}
Note that if (1) holds then $Z_5$ has full rank because all of our previous conditions for rank drop require at least the $x_i$'s or the $y_i$'s to be in a line. 

The variety $\PP^2 \times \PP^2$ can be embedded in $\PP^8$ as the 
image of the {\bf Segre map}
\begin{align}\label{eq:Segre map}
\PP^2 \times \PP^2 \rightarrow \PP^8 \,\,\,\textup{ such that } \,\,\,
(x,y) \mapsto x^\top \otimes y^\top.
\end{align} 
In what follows it will be helpful to identify $\PP^8$ with the set of all $3 \times 3$ complex matrices up to scale via the bijection $M \leftrightarrow \textup{vec}(M)$. 
Then the Segre embedding of $\PP^2 \times \PP^2$ in $\PP^8$  (which we will also call $\PP^2 \times \PP^2$), can be identified with the set of rank one $3\times 3$
matrices up to scale since 
\begin{align}
x^\top \otimes y^\top = (\textup{vec}(yx^\top))^\top.
\end{align}
The Segre variety $\PP^2 \times \PP^2 \subset \PP^8$ 
has dimension $4$ and degree $6$. It is the zero locus 
of the determinantal ideal generated by the $2 \times 2$ minors of a symbolic $3 \times 3$ matrix whose entries denote the coordinates of $\PP^8$. 

We first show that 
given $5$ points $\{(x_i,y_i)\}_{i=1}^5$ in general position, there exists a unique $6$th point $(x_6,y_6) \in \PP^2 \times \PP^2$ such that $\{(x_i,y_i)\}_{i=1}^6$ is rank deficient.  

\begin{lem} \label{lem:6th point pair}
Let $\{(x_i,y_i)\}_{i=1}^5$ be in general position. Then there is a unique new $6$th point $(x_6,y_6) \in \PP^2 \times \PP^2$ such that the matrix $Z_6$ is rank deficient. This $6$th point is a rational function of the first $5$; if we use homographies to fix the first $4$ points as
\begin{align*}
x_1=(1,0,0)=y_1\quad\quad x_2=(0,1,0)=y_2\\
x_3=(0,0,1)=y_3\quad\quad x_4=(1,1,1)=y_4, 
\end{align*}
then the $6$th point has the formula:
\begin{equation}\label{eq:rational formulas}
\begin{split}
x_6=&(\frac{y_{53}-y_{52}}{y_{53}x_{52}-x_{53}y_{52}},\frac{y_{53}-y_{51}}{y_{53}x_{51}-x_{53}y_{51}},\frac{y_{51}-y_{52}}{y_{51}x_{52}-x_{51}y_{52}})\\
y_6=&(\frac{x_{53}-x_{52}}{y_{53}x_{52}-x_{53}y_{52}},\frac{x_{53}-x_{51}}{y_{53}x_{51}-x_{53}y_{51}},\frac{x_{51}-x_{52}}{y_{51}x_{52}-x_{51}y_{52}}).
\end{split}
\end{equation}
\end{lem}

\begin{proof} 
Let $\mathcal{R}(Z_5)$ denote the row span of $Z_5 = 
(x_i \otimes y_i)_{i=1}^5$. By general position, 
$\rank(Z_5) = 5$ which means that 
 $\mathcal{R}(Z_5) \cong
\PP^4$. Therefore, generically, $\mathcal{R}(Z_5)$ intersects the $4$-dimensional, degree $6$ Segre variety $\PP^2 \times \PP^2$ transversally in $6$ distinct points.  Each of these points corresponds to a rank one matrix $yx^\top$ which we have identified with 
$x^\top \otimes y^\top$. Five of these intersection points are the matrices $x_i^\top \otimes y_i^\top$ for $i=1, \ldots, 5$ and the $6$th point $x_6^\top \otimes y_6^\top$ corresponds to a unique point $(x_6,y_6) \in \PP^2 \times \PP^2$.

To prove the second part of the Lemma, we first use homographies to fix the first $4$ points as mentioned. We can then use Macaulay2 to verify that we can solve for $(x_6,y_6)$ in terms of $(x_5,y_5)$ as in equations \eqref{eq:rational formulas}.
\end{proof}

The existence of, and constructions for, this $6$th point were noted in both \cite{chiral} and \cite{wernerfivepoints}, but not the rational formula. The rational formula shows that if $(x_i,y_i)_{i=1}^5$ are real then $(x_6,y_6)$ will be real as well. 
Note that if $\{x_i\}_{i=1}^5$ and $\{y_i\}_{i=1}^5$ were related by homography then the denominators in \eqref{eq:rational formulas} would all be $0$. Here is an illustration of Lemma~\ref{lem:6th point pair}. 

\begin{ex}\label{ex:rational formula part one}
We begin with $5$ point pairs
\begin{align*}
x_1&=(1,0,0)\quad y_1=(1,0,0) &\quad\quad
x_2&=(0,1,0)\quad y_2=(0,1,0)\\
x_3&=(0,0,1)\quad y_3=(0,0,1) &\quad\quad
x_4&=(1,1,1)\quad y_4=(1,1,1)\\
x_5&=(3,5,1)\quad y_5=(8,2,1)
\end{align*}
If we leave $(x_6,y_6)$ symbolic and construct the matrix $Z_6$, we can verify with Macaulay2 that the unique new point pair that will result in rank deficiency is $x_6\sim(-\frac{1}{3},\frac{7}{5},\frac{3}{17}),~ y_6\sim(-\frac{4}{3},\frac{2}{5},-\frac{1}{17})$, which is exactly what we get from the formula above.
\end{ex}

In Lemma~\ref{lem:counterpart rational function} we will prove a counterpart to the above result, where we instead begin with $x_1,\ldots,x_6$ and calculate (up to homography) the points $y_1,\ldots,y_6$ that will result in rank deficiency.

\subsection{Algebraic Characterizations of Rank Deficiency}\label{sec:k=6 algebra}
In this section we will write down various algebraic characterizations of rank deficiency. While these statements are by no means obvious or intuitive, once stated, they are straightforward to check using Macaulay2.  In the next section we will explain the geometric  origins of these statements. 

Our first result is a characterization of rank drop using brackets and cross-ratios as in the previous sections. 
\begin{defn} \label{def:cross-ratios ideal} 
Given $6$ input points $(x_i,y_i) \in 
\PP^2 \times \PP^2$, 
let $\mathcal{B}_{rackets}(Z_6)$ be the ideal generated by the differences of the two sides in the following bracket equations: 
\begin{equation} \label{eq:k=6 brackets}
[ikp]_x[jrp]_x[irs]_y[jks]_y=[irp]_x[jkp]_x[iks]_y[jrs]_y
\end{equation}
for all distinct indices $i,j,k,r,p,s\in\{1,2,3,4,5,6\}$. 
\end{defn}

Up to permutations, there are $30$ equalities in \eqref{eq:k=6 brackets}: one for each choice of $p$ and $s$. Under sufficient genericity, the bracket 
equations are the planar cross-ratio equalities:
\begin{equation}\label{eq:k=6 planar cross-ratio conditions}
(x_i,x_j;x_k,x_r;x_p)=(y_i,y_j;y_k,y_r;y_s). 
\end{equation}

Next we define a specific kind of highly degenerate configuration.

\begin{defn}
We say that $\{(x_i,y_i)\}_{i=1}^6$ is an \textbf{asymmetric double triangle} if the $\{x_i\}$ and $\{y_i\}$ each consist of $3$ distinct points such that, up to reordering, $y_1=y_2$, $y_3=y_4$, $y_5=y_6$, $x_1=x_4$, $x_2=x_5$, and $x_3=x_6$. 
\end{defn}
\begin{lem}\label{lem:dtriangleref}
Let $\{(x_i,y_i)\}_{i=1}^6$ be an asymmetric double triangle. Then $Z_6$ is rank deficient if and only if an inherited condition holds.
\end{lem}
\begin{proof}
We may assume up to reordering that $y_1=y_2:=v_1$, $y_3=y_4:=v_2$, $y_5=y_6:=v_3$, $x_1=x_4:=u_1$, $x_2=x_5:=u_2$, and $x_3=x_6:=u_3$. If $u_1,u_2,u_3$ and $v_1,v_2,v_3$ are both non-collinear, then we can assume by change of coordinates that $u_1=v_1=(1,0,0)$, $u_2=v_2=(0,1,0)$, and $u_3=v_3=(0,0,1)$. It is then simple to check that $Z_6$ has full rank.

If the three of points of one type are distinct but collinear and the three points of the other type are non-collinear, then we may assume up to change of coordinates and relabeling that $u_1=v_1=(1,0,0)$, $u_2=v_2=(0,1,0)$, $u_3=(1,1,0)$, and $v_3=(0,0,1)$. Again, it is simple to check that $Z_6$ has full rank.

If two points of one type are coincident, say $u_1=u_2$, then an inherited condition holds and rank drops. If $u_1,u_2,u_3$ and $v_1,v_2,v_3$ are both collinear then again rank drops due to an inherited condition. 
\end{proof}
Let $\mathcal{M}_{inors}(Z_6)$ denote the ideal of maximal minors of $Z_6$. The following statement can be checked using Macaulay2. 

\begin{thm}\label{thm:k=6 neither side in a line}
Let $\{(x_i,y_i)\}_{i=1}^6$ be such that neither the $x_i$'s nor the $y_i$'s are in a line and 
$\{(x_i,y_i)\}_{i=1}^6$ is not an asymmetric double triangle. Then $Z_6$ is rank deficient if and only if either one of the inherited conditions hold or
the $30$ bracket equations \eqref{eq:k=6 brackets} hold.
\end{thm}

\begin{proof}
As in the proof of Theorem \ref{thm:k=5} we split into two cases. Applying homographies we can assume either that
\begin{align*}
x_1=(1,0,0)=y_1\quad\quad x_2=(0,1,0)=y_2\quad\quad x_3=(0,0,1)=y_3. 
\end{align*}
or
\begin{align*}
x_1=(1,0,0)=y_1\quad\quad x_2=(0,1,0)=y_2\quad\quad x_3=(0,0,1)=y_4. 
\end{align*}
In the first case the ideal $\mathcal{M}_{inors}(Z_6)$ decomposes into $19$ components. The first $6$ correspond to invalid inputs $x_i,y_i=(0,0,0)$, the next $12$ correspond to repeated point pairs, and the $19$th component we refer to as $M_{19}$. The ideal $\mathcal{B}_{rackets}(Z_6)$ decomposes into $13$ components. One 
component is $M_{19}$ and the other $12$ components correspond to asymmetric double triangles.

In the second case the ideal $\mathcal{M}_{inors}(Z_6)$ decomposes into $17$ components. The first $6$ correspond to invalid inputs $x_i,y_i=(0,0,0)$, the next $10$ correspond to repeated point pairs, and the $17$th component we refer to as $M_{17}$. The ideal $\mathcal{B}_{rackets}(Z_6)$ decomposes into $21$ components. One 
component is $M_{17}$ and the other $20$ components correspond to asymmetric double triangles.
\end{proof}

In Theorem~\ref{thm:k=6 neither side in a line} the only assumption is that neither the $x_i$'s nor the $y_i$'s are in a line. 
In Theorem~\ref{thm: k=6 one side in a line} we will tackle the case of one side in a line, and together the two theorems 
cover all possibilities. 

Before we can do that, we consider an important subcase of~\Cref{thm:k=6 neither side in a line}, namely the generic case. The main result here is~\Cref{thm:k=6 general position} which uses tools from invariant theory that we have not encountered thus far. Besides being a key result of this paper, it 
will also  lead to a proof of~\Cref{thm: k=6 one side in a line},  completing the story. 

Suppose $x_1, \ldots, x_6 \in \PP^2_x$ and $y_1, \ldots, y_6 \in \PP^2_y$ are 
in general position. Then we may fix 
\begin{align} \label{eq:fixing}
\begin{split}
x_1 = y_1 = (1,0,0), \quad\quad x_2 = y_2 = (0,1,0), \\
x_3 = y_3 = (0,0,1), \quad\quad x_4=y_4 = (1,1,1)
\end{split}
\end{align}
and consider the rank deficiency of 
$Z_6$ whose $5$th and $6$th rows are left symbolic as:
\begin{align} \label{eq:rows of Z6}
\begin{split}
    (x_{51},x_{52},x_{53})\otimes(y_{51},y_{52},y_{53}),(x_{61},x_{62},x_{63})\otimes(y_{61},y_{62},y_{63}).
    \end{split}
\end{align}
Recall that the general position assumption guarantees that 
no submatrix of $Z_6$ with $5$ rows is rank deficient.  
Under these assumptions, one can use Macaulay2 to check that $\mathcal{M}_{inors}(Z_6)$ is a radical ideal of dimension $9$ and degree $4$. Its prime decomposition consists of $14$ components. The first $4$ of these components correspond to invalid inputs $x_i,y_i=(0,0,0)$, the next $9$ correspond to repeated points, and the $14$th component, denoted as $M_{14}$, encodes a 
new set of conditions, and has dimension $8$ and degree $35$. Since $Z_6$ is rank deficient if and only if all its maximal minors vanish,  we conclude that the new conditions for rank drop 
are exactly that $(x_5,y_5)$ and $(x_6,y_6)$ lie in the variety of $M_{14}$.

We will now interpret $M_{14}$ in terms of the invariants of $x_1, \ldots, x_6$ and $y_1, \ldots, y_6$. 
For a set of $6$ points $p_1,\ldots,p_6 \in \PP^2$ define 
\begin{equation}[(ij)(kl)(rs)]:=[ijr][kls]-[ijs][klr].\end{equation}
This is a classical invariant of points in $\PP^2$ 
whose vanishing expresses that the three lines 
$\overline{p_ip_j}$, $\overline{p_kp_l}$ and $\overline{p_rp_s}$ 
meet in a point \cite{coble}*{pp 169}. Using these invariants, Coble 
defines the following $6$ scalars \cite{coble}*{pp 170}:
\begin{small}
\begin{align}\label{eq:bumped up Joubert invariants}
\begin{split}
\bar{a}&=[(25)(13)(46)]+[(51)(42)(36)]+[(14)(35)(26)]+[(43)(21)(56)]+[(32)(54)(16)]\\
\bar{b}&=[(53)(12)(46)]+[(14)(23)(56)]+[(25)(34)(16)]+[(31)(45)(26)]+[(42)(51)(36)]\\
     \bar{c}&=[(53)(41)(26)]+[(34)(25)(16)]+[(42)(13)(56)]+[(21)(54)(36)]+[(15)(32)(46)]\\
\bar{d}&=[(45)(31)(26)]+[(53)(24)(16)]+[(41)(25)(36)]+[(32)(15)(46)]+[(21)(43)(56)]\\
\bar{e}&=[(31)(24)(56)]+[(12)(53)(46)]+[(25)(41)(36)]+[(54)(32)(16)]+[(43)(15)(26)]\\
\bar{f}&=[(42)(35)(16)]+[(23)(14)(56)]+[(31)(52)(46)]+[(15)(43)(26)]+[(54)(21)(36)]
\end{split}
\end{align}
\end{small}

Let $\bar{a}_x, \bar{b}_x, \bar{c}_x, \bar{d}_x, \bar{e}_x, \bar{f}_x$ be the polynomials 
obtained by computing the expressions in 
\eqref{eq:bumped up Joubert invariants} using 
$6$ symbolic vectors $x_1, \ldots, x_6$ where each $x_i = (x_{i1},x_{i2},x_{i3})^\top$. Similarly, let 
$\bar{a}_y, \bar{b}_y, \bar{c}_y, \bar{d}_y, \bar{e}_y, \bar{f}_y$ be  obtained from symbolic $y_1, \ldots, y_6$ where each $y_i = (y_{i1},y_{i2},y_{i3})^\top$. 

\begin{defn} \label{def:2x2 minors}
Let $\mathcal{I}_{nvariants}(Z_6)$ denote the ideal of $2 \times 2$ minors of 
\begin{align} \label{eq:2x6 matrix}
    \begin{bmatrix} \bar{a}_x &  \bar{b}_x &  \bar{c}_x &  \bar{d}_x &  \bar{e}_x &  \bar{f}_x\\
    \bar{a}_y &  \bar{b}_y &  \bar{c}_y & \bar{d}_y &  \bar{e}_y &  \bar{f}_y\end{bmatrix} 
\end{align}
after specializing the first $4$  points $(x_i,y_i)$ as in \eqref{eq:fixing}.
\end{defn}

The ideal $\mathcal{I}_{nvariants}(Z_6)$ is radical of dimension $9$ and degree $4$. Its prime decomposition consists of $14$ components. The first $4$ of these components correspond to invalid inputs $x_i,y_i=(0,0,0)$, the next $8$ correspond to $3$ repeated points in either $\PP^2_x$ or $\PP^2_y$, the $13$th component corresponds to $x_5\sim y_5$ and $x_6\sim y_6$ (which encodes that $\{x_i\}$ and $\{y_i\}$ are related by a homography). The $14$th component, denoted as $I_{14}$, provides a new condition. This ideal has dimension $8$ and degree $35$.

\begin{thm} \label{thm:k=6 general position}
Given a configuration $\{(x_i,y_i)\}_{i=1}^6$ such that the $\{x_i\}$ and $\{y_i\}$ are in general position, 
\begin{align} \label{eq:bar vectors are multiples of each other}
 \begin{pmatrix} \bar{a}_x &  \bar{b}_x &  \bar{c}_x &  \bar{d}_x &  \bar{e}_x &  \bar{f}_x
 \end{pmatrix} \sim
  \begin{pmatrix} 
    \bar{a}_y &  \bar{b}_y &  \bar{c}_y & \bar{d}_y &  \bar{e}_y &  \bar{f}_y \end{pmatrix}
\end{align}
if and only if either the matrix $Z_6$ is rank deficient or $y_i = Hx_i$ for some homography $H \in \textup{PGL}(3)$.
\end{thm}

\begin{proof}
It can be checked using Macaulay2 that $M_{14}=I_{14}$.
\end{proof}

The statement of Theorem~\ref{thm:k=6 general position} begs the question of how homography ties in with rank drop, if at all. We give an example to show that the existence of a homography such that $y_i = Hx_i$ for $6$ points 
$(x_i,y_i)$ is not sufficient for $Z_6$ to be 
rank deficient. In contrast, it does suffice for $7$ point pairs. 
\begin{ex}\label{ex:homography bad}
By Lemma \ref{lem:rank drop is invariant under homographies}, if there is a $H \in \textup{PGL}(3)$ such that $y_i = Hx_i$ for $i=1,\ldots, 6$, then we may assume that $x_i=y_i$ for $i=1,\ldots,6$. Then the null space of $Z_6$ consists exactly of the vectors $\textup{vec}(F)$ for the $3 \times 3$ matrices 
$F$ such that
\begin{equation}
    x_i^\top Fx_i=0, \,\,\forall\,\, i = 1, \ldots, 6.
\end{equation}
 Generically, this is exactly the set of $3 \times 3$ skew-symmetric matrices, $\text{Skew}_3 \cong \PP^2$. 
 Hence $\rank(Z_6) = 6$. 
 For example, consider
\begin{align*}
x_1=y_1=(1,0,0)\quad\quad x_2=y_2=(0,1,0)\quad\quad x_3=y_3=(0,0,1)\\
x_4=y_4=(1,1,1)\quad\quad x_5=y_5=(2,3,1)\quad\quad x_6=y_6=(3,7,1)
\end{align*}
for which the nullspace of $Z_6$ is $\text{Skew}_3$. In general, if the homography is given by the matrix $H$ then the null space of $Z_6$ is obtained by left action of $H^\top$ on $\text{Skew}_3$. We note that this null space is unusual in that will always consist entirely of rank $2$ matrices.
\end{ex}





To finish our algebraic discussion, we consider the case of one $\PP^2$ having all input points in a line.
In order to state the result, we introduce the classical {\bf Joubert invariants} of 
$6$ points in $\PP^1$ \cite{coble} which are: 
\begin{small}
\begin{align} \label{eq:Joubert invariants}
\begin{split}
A_y&=[25][13][46]+[51][42][36]+[14][35][26]+[43][21][56]+[32][54][16]\\
B_y&=[53][12][46]+[14][23][56]+[25][34][16]+[31][45][26]+[42][51][36]\\
C_y&=[53][41][26]+[34][25][16]+[42][13][56]+[21][54][36]+[15][32][46]\\
D_y&=[45][31][26]+[53][24][16]+[41][25][36]+[32][15][46]+[21][43][56]\\
E_y&=[31][24][56]+[12][53][46]+[25][41][36]+[54][32][16]+[43][15][26]\\
F_y&=[42][35][16]+[23][14][56]+[31][52][46]+[15][43][26]+[54][21][36].
\end{split}
\end{align}
\end{small}
\begin{thm}\label{thm: k=6 one side in a line}
Suppose $y_1, \ldots, y_6 \in \PP^2_y$ are in a line. Then $Z_6$  is rank deficient if and only if
\begin{equation}\label{eq: k=6 one side in a line basic eq}
\bar{a}_xA_y+\bar{b}_xB_y+\bar{c}_xC_y+\bar{d}_xD_y+\bar{e}_xE_y+\bar{f}_xF_y=0
\end{equation}
Furthermore, if $x_1,\ldots,x_6$ are in general position and $y_1,\ldots,y_6$ are distinct, then \eqref{eq: k=6 one side in a line basic eq} holds if and only if there exists a projection $T:\PP^2_x\to\PP^1_y$ such that $T(x_i)\sim y_i$ for all $i$.
\end{thm}

This theorem is analogous to Theorems \ref{thm:k=2,3,4} (3)(b) and \ref{thm: k=5 geometry theorem}. The proof of this theorem will follow from the geometry that we develop in the next section. 

\subsection{Geometric Characterizations of Rank Deficiency}
\label{sec:k=6 geometry}
We next explain where the algebraic results (and ideals) from the previous section come from by providing a geometric characterization of  rank deficiency of $Z_6$. This relies on
the classical theory of smooth cubic surfaces in $\PP^3$. For the convenience of the reader, we have collected the needed facts about cubic surfaces in Appendix~\ref{sec:cubic surface facts}.

\subsubsection{Cubic surfaces and Sch\"afli double sixes}

A cubic surface in $\PP^3$ is a hypersurface defined by a single homogenous polynomial of degree $3$ in $\CC[z_0,z_1,z_2,z_3]$ where $(z_0,\ldots,z_3)$ are the coordinates of $\PP^3$. 

\begin{defn} \label{def:SD6}
A pair of 6 lines $\{\ell_1, \ldots, \ell_6\}$ and 
$\{\ell_1', \ldots, \ell_6'\}$ in $\PP^3$ is called a {\bf Schl\"afli double six} if 
\begin{enumerate}
    \item the $6$ lines in each set are mutually skew (i.e., they don't intersect pairwise), and \item $\ell_i$ intersects $\ell_j'$ if and only if $i \neq j$.
    \end{enumerate}
\end{defn}
A smooth cubic surface has $36$ Schl\"afli double sixes. 
The rich and highly symmetric properties of these line configurations are listed in Appendix~\ref{sec:cubic surface facts} and lie at the heart of all of our major results in this section.
Our main (geometric) theorem in this section is the following. 

\begin{thm}\label{thm:k=6 central theorem}
Suppose $\{(x_i,y_i)\}_{i=1}^6 \subset \PP^2 \times \PP^2$ is in general position. Then $Z_6$ is 
 rank  deficient if and only if there exists a cubic surface $\mathcal{S}$ that can be obtained as the 
 blow up of both $\PP^2_x$ in $x_1, \ldots, x_6$ and $\PP^2_y$ in
 $y_1, \ldots, y_6$ such that their respective 
 exceptional curves 
$\{\ell_1,\ldots,\ell_6\}, \{\ell_1',\ldots,\ell_6'\}$
form a Schl\"afli double six on $\mathcal{S}$. 
When $Z_6$ is rank deficient, this cubic surface has the determinantal representation $\mathcal{N}(Z_6)\cap\mathcal{D}$, where  $\mathcal{N}(Z_6)$ denotes the right null space of $Z_6$ and $\mathcal{D}$ is the cubic hypersurface of all rank deficient $3\times 3$ matrices.
\end{thm}

To prepare for the proof of this theorem, we begin by constructing a cubic surface from $5$ points $\{(x_i,y_i)\}_{i=1}^5$ in general position. 
Recall that we can identify a point in $\PP^8$ with a $3 \times 3$ matrix $M$ via its vectorization 
$\textup{vec}(M)$. Under this identification, $\mathcal{D} \subset \PP^8$ is the cubic hypersurface $\mathcal{D} = \{ M \,:\, \det(M) = 0\}$. 
Let $\mathcal{N}(Z_5) \subset \PP^8$ be the null space of $Z_5$. Since $\{(x_i,y_i)\}_{i=1}^5$ is in general position, $\rank(Z_5) = 5$ and so, $\mathcal{N}(Z_5) \cong 
\PP^3$. Therefore,  
\begin{align} \label{eq:model1}
    S := \mathcal{N}(Z_5) \cap \mathcal{D}
\end{align}
is a cubic surface in 
$\mathcal{N}(Z_5) \cong \PP^3 \subset \PP^8$. This surface is smooth by the general position assumption. (We note that if  $\{x_i\}$ and $\{y_i\}$ were related by a homography then $S$ would be degenerate (see Example \ref{ex:homography bad}).)

The surface $S$ has a natural {\em determinantal representation} (see Section~\ref{sec:cubic surface facts} for the definition): Pick a basis $M_0, \ldots, M_3$ of $\mathcal{N}(Z_5)$ and set 
\begin{align} \label{eq:matrix pencil}
    M(z) = z_0 M_0 + \cdots + z_3 M_3
\end{align}
where $z_i \in \CC$. If $(x_i,y_i),i=1,\ldots,5$ are real, we can choose $M_0, \ldots, M_3$ to be real matrices. 
Since $M(z)$ is a parameterization of $\mathcal{N}(Z_5)$, $S$ is cut out by $\det(M(z)) = 0$, and $M(z)$ is a determinantal representation of $S$.
Generically, the $3$-dimensional $\mathcal{N}(Z_5)$ does not intersect the $4$-dimensional $\PP^2 \times \PP^2$ of rank one matrices, and so a point $M(z)$ on $S$ has rank exactly $2$. This also follows from $S$ being smooth.

Next we construct a Schl\"afli double six  on the cubic surface $S = \mathcal{N}(Z_5) \cap \mathcal{D}$, from $\{(x_i,y_i)\}_{i=1}^5$. 

\begin{defn} Given $\{(x_i,y_i)\}_{i=1}^5$ in general position, define 
\begin{align}
    \ell_{x_i} := & \{ M(z) \,:\,  M(z) x_i= 0 \} \textup{ for } i = 1, \ldots,5, \textup{ and } \\
    \ell_{y_j}' := & \{ M(z) \,:\, y_j^\top M(z) = 0 \} \textup{ for } j=1,\ldots,5,
\end{align}
where $M(z)$ is of the form \eqref{eq:matrix pencil}.
\end{defn}

For each $i=1,\ldots,5$, $\ell_{x_i} \subset \mathcal{N}(Z_5)$ by the definition of $M(z)$ in \eqref{eq:matrix pencil}. Further, since  $ M(z) x_i =0$,  $M(z)$ has a non-trivial right null space and hence,  $\det(M(z))=0$. Therefore, 
$\ell_{x_i} \subset S = \mathcal{N}(Z_5) \cap \mathcal{D}$. 
The condition that $M(z) \in \mathcal{N}(Z_5)$ is equivalent to $M(z)$ satisfying the $5$ linear constraints 
\begin{align}
    \langle x_i^\top \otimes y^\top_i, M(z) \rangle = y_i^\top M(z) x_i = 0 \textup{ for } i=1,\ldots, 5.
\end{align}
Therefore, $\ell_{x_i}$ is cut out by $8$ linear constraints including the $3$ constraints given by $ M(z) x_i= 0$.
However one of them is redundant since $ M(z) x_i = 0$ implies that $y_i^\top M(z) x_i = 0$. Therefore, $\ell_{x_i}$ is cut out by $7$ independent linear constraints and is hence a line on the cubic surface $S$. Similarly, 
each $\ell_{y_j}', j=1,\ldots,5$  is a line on $S$.

\begin{lem} \label{lem:5 pairs give a Sd6}
The lines $\ell_{x_i}$ and $\ell_{y_j}'$ satisfying the following incidence relations:
\begin{enumerate}
    \item The lines $\{ \ell_{x_i}, i=1,\ldots,5\}$ are mutually skew.
    \item The lines $\{\ell_{y_j}', j = 1, \ldots, 5\}$ are mutually skew. 
    \item $\ell_{x_i}$ intersects $\ell_{y_j}'$ if and only if $i \neq j$.
\end{enumerate}
Therefore,  $\{ \ell_{x_i}, i=1,\ldots,5\}$ and  $\{\ell_{y_j}', j = 1, \ldots, 5\}$ form part of a unique Schl\"afli double six on $S$.
\end{lem}

\begin{proof}
The proof of the incidence relations follows from dimension counting and can be found in \cite{chiral}*{Lemma~6.1}
The two sets of lines are part of a unique Schl\"afli double six on $S$  since each double six is already uniquely defined by three of its line pairs, and we have $5$ line pairs here with the correct incidences (see Lemma~\ref{lem:cubic surface facts 1} (3)).
\end{proof}


By Lemma \ref{lem:cubic surface facts 1}, given a Schl\"afli double six on $S$, there are two sets of $6$ points 
$\{p_i\} \subset \PP^2_p$ and $\{q_j\} \subset \PP^2_q$ such that $S$ is the blow up of $\PP^2_p$ at $\{p_i\}$ with blow up morphism $\pi_p \,:\, S \rightarrow \PP^2_p$,  and  also the blow up of $\PP^2_q$ at $\{q_j\}$ with blow up morphism $\pi'_q \,:\, S \mapsto \PP^2_q$. Each set of skew lines in the given double six is the set of exceptional curves of these 
point sets under $\pi_p$ and $\pi'_q$. 
The blow up morphism $\pi_p \,:\, S \rightarrow \PP^2_p$ sends $M(z) \in S$ to the unique $p \in \PP^2$ in its right null space and the blow up morphism $\pi'_q \,:\, S \rightarrow \PP^2_q$ sends $M(z) \in S$ to the unique $q \in \PP^2$ in its left null space. See the second half of Appendix~\ref{sec:cubic surface facts}. 

We are now ready to prove the forward direction of Theorem~\ref{thm:k=6 central theorem}. 


\begin{defn}\label{def:6th line pair}
Let $\ell_6, \ell_6'$ denote the $6$th pair of lines in the Schl\"afli double six on $S$ consisting of $\{ \ell_{x_i} \}_{i=1}^5$ and $\{\ell_{y_j}'\}_{j=1}^5$ 
from Lemma~\ref{lem:5 pairs give a Sd6}.
\end{defn}

\begin{lem} \label{lem:Sd6 characterization of rank drop}
Suppose $\{(x_i,y_i)\}_{i=1}^5$ is in general position. Then the addition of a new row $x^\top_6 \otimes y^\top_6$ makes $Z_6$ rank deficient if and only if
 $\{\ell_{x_i}\}_{i=1}^6$ and $\{\ell_{y_j}'\}_{j=1}^6$ form a Schl\"afli double six on $S$ where 
\begin{align}
\ell_6=\ell_{x_6} = \{ M(z) \,:\,  M(z) x_6 = 0 \}\quad\textup{ and }\quad 
\ell'_6=\ell_{y_6}' = \{M(z) \,:\, y_6^\top M(z) = 0 \}.
\end{align}
\end{lem}

\begin{proof}
To prove the forward direction, suppose $Z_6$ is rank deficient. Then 
$\mathcal{N}(Z_5) = \mathcal{N}(Z_6)$ and 
$S = \mathcal{N}(Z_6) \cap \mathcal{D}$. Check that $\ell_{x_6}$ is a line 
on $S$ by the same argument as for $\ell_{x_i}, i \leq 5$. Similarly, 
$\ell_{y_6}'$ is a line on $S$. One can also check that the collection 
$\{\ell_{x_i}\}_{i=1}^6$ and $\{\ell_{y_j}'\}_{j=1}^6$ satisfies the incidence relations of a Schl\"afli double six. This must mean that $\ell_6 = \ell_{x_6}$ and 
$\ell_6' = \ell_{y_6}'$ since the first $5$ line pairs determine a unique double six on $S$. 

To prove the reverse direction, suppose we have a Schl\"afli double six on $S$ as in the statement of the lemma. Then by the discussion before Definition~\ref{def:6th line pair}, $S$ is the blow up of $\PP^2_x$ at $\{x_i\}$ under a morphism $\pi_x$ and $S$ is also the blow up of $\PP^2_y$ at $\{y_i\}$ under a morphism $\pi_y'$. Further, $\pi_x(\ell_{x_i}) = x_i$ for $i=1, \ldots, 6$ and 
$\pi'_y(\ell_{y_j}') = y_j$ for $j=1, \ldots, 6$. We need to argue that 
$Z_6 =(x^\top_i \otimes y^\top_i)_{i=1}^6$ is rank deficient. 
Recall Lemma \ref{lem:6th point pair} and let $(x_0,y_0)$ be the unique new point pair such that $Z'$ with rows $x_i^\top \otimes y_i^\top$ for $i=0,\ldots,5$ is rank deficient. Then, by the forwards direction, we must have $\ell_6$ be the exceptional curve of $x_0$ under $\pi_x$ and hence $x_0=\pi_x(\ell_6)=x_6$. Similarly, 
$y_0=\pi_y(\ell'_6)=y_6$. We conclude that $(x_0,y_0)=(x_6,y_6)$ and therefore, $Z_6$ is rank deficient.
\end{proof}

The following corollary to Lemma \ref{lem:Sd6 characterization of rank drop} and Theorem \ref{thm:summary of cubic surface facts}
concludes the proof of the forward direction of Theorem~\ref{thm:k=6 central theorem}.

\begin{cor}\label{cor:rank deficiency gives SD6}
If $\{(x_i, y_i)\}_{i=1}^6$ is in general position and $Z_6$ is rank deficient 
then the cubic surface 
$S= \mathcal{N}(Z_6) \cap \mathcal{D}$ is the blow up of  $\PP^2_x$ at $x_1, \ldots, x_6$ with exceptional curves $\ell_{x_i},i=1,\ldots,6$, and also the blow up of $\PP^2_y$ at 
$y_1, \ldots,y_6$ with exceptional curves $\ell_{y_j}',j=1,\ldots,6$. The two sets 
of lines form a Schl\"afli double six on $S$.
\end{cor}

Note that Lemma~\ref{lem:Sd6 characterization of rank drop} gives a 
second construction for the unique point pair 
$(x_6,y_6)$ in Lemma~\ref{lem:6th point pair}. Namely, as the images under the blow up morphisms, of the final pair of lines in the unique Schl\"afli double six on $S$ determined 
by the lines 
$\{\ell_{x_i}\}_{i=1}^5$ and $\{\ell_{y_j}'\}_{j=1}^5$. This was also noted in \cite{chiral}.

\medskip 

To prove the backward direction of Theorem~\ref{thm:k=6 central theorem}, we will first prove a counterpart to Lemma \ref{lem:6th point pair}.
\begin{lem}\label{lem:counterpart rational function}
If $x_1,\ldots,x_6$ are in general position then there is a unique set of $6$ points $y_1,\ldots,y_6$ in general position (up to homography) such that $Z_6$ is rank deficient. If we assume that
\begin{align*}
x_1=(1,0,0)=y_1\quad\quad x_2=(0,1,0)=y_2\\
x_3=(0,0,1)=y_3\quad\quad x_4=(1,1,1)=y_4
\end{align*}
then we can obtain $y_5,y_6$ via the rational function
\begin{align}
y_5=&(\frac{x_{63}-x_{62}}{x_{53}x_{62}-x_{63}x_{52}},\frac{x_{63}-x_{61}}{x_{53}x_{61}-x_{63}x_{51}},\frac{x_{62}-x_{61}}{x_{61}x_{52}-x_{51}x_{62}})\\
y_6=&(\frac{x_{53}-x_{52}}{x_{53}x_{62}-x_{63}x_{52}},\frac{x_{53}-x_{51}}{x_{53}x_{61}-x_{63}x_{51}},\frac{x_{52}-x_{51}}{x_{61}x_{52}-x_{51}x_{62}}).
\end{align}
\end{lem}
\begin{proof}
We first prove that this construction will result in rank deficiency; the uniqueness will come afterwards as a consequence of Corollary \ref{cor:rank deficiency gives SD6}. We construct the matrix $Z_6=(x^\top_i\otimes y^\top_i)_{i=1}^6$ and fix $x_1,\ldots x_4,y_1,\ldots y_6$ as in the statement of the Lemma. We can then verify with Macaulay2 that all $6\times 6$ minors of $Z_6$ are $0$ and therefore $Z_6$ is rank deficient.

To see that this choice is unique, construct $S=\mathcal{N}(Z_6)\cap\mathcal{D}$. Then by Corollary \ref{cor:rank deficiency gives SD6}, $S$ is both the blow up of $\PP^2_x$ at $x_1,\ldots,x_6$ and the blow up of $\PP^2_y$ at $y_1,\ldots,y_6$ and the two sets of exceptional curves $\{\ell_{x_1},\ldots,\ell_{x_6}\},\{\ell'_{y_1},\ldots,\ell'_{y_6}\}$ form a Schl\"afli double six. Now suppose there was some other $y_1',\ldots,y_6'$ in general position such that $Z'_6=(x_i^{\top} \otimes {y'_i}^\top)_{i=1}^6$ was rank deficient.  Then $S'=\mathcal{N}(Z'_6)\cap\mathcal{D}$ is also the blow up of $\PP^2_x$ at $x_1,\ldots,x_6$. Moreover, it is also the blow up of $\PP^2_y$ at $y_1',\ldots,y_6'$ and the exceptional curves $\ell_{x_1},\ldots,\ell_{x_6},\ell'_{y'_1},\ldots,\ell'_{y'_6}$ form a Schl\"afli double six. It follows by Lemma \ref{lem:double six facts} that $\ell'_{y_i}=\ell'_{y'_i}$ for all $i$. If we compose the blow up morphisms $\pi_{y'}\circ\pi_y^{-1}:\PP^2_y\to\PP^2_y$ we obtain an invertible projective transformation; thus this composition is a homography taking $y_i\mapsto y_i'$. This concludes the proof. 
\end{proof}

The formula shows that if $x_1,\ldots,x_6$ are real then, up to homography, $y_1,\ldots,y_6$ will also be real.

\begin{ex}\label{ex:rational formula part two}
Consider Example \ref{ex:rational formula part one} again and now fix the points
\begin{align*}
x_1&=(1,0,0)\quad x_2=(0,1,0) &\quad\quad y_1&=(1,0,0)
\quad y_2=(0,1,0)\\
x_3&=(0,0,1)\quad x_4=(1,1,1) &\quad\quad
y_3&=(0,0,1)\quad y_4=(1,1,1)\\
x_5&=(3,5,1)\quad x_6=(-\frac{1}{3},\frac{7}{5},\frac{3}{17}).
\end{align*}
If we leave $y_5,y_6$ symbolic and construct the matrix $Z_6$, we can verify with Macaulay2 that the unique new points that will result in rank deficiency are $y_5\sim(8,2,1),~ y_6\sim(-\frac{4}{3},\frac{2}{5},-\frac{1}{17})$, which is exactly what we get from the formula above. Moreover, check that this matches up exactly with the points in Example \ref{ex:rational formula part one}.
\end{ex}

The following corollary now completes the reverse direction of Theorem~\ref{thm:k=6 central theorem}.

\begin{cor} \label{cor:Sd6 gives rank deficiency}
Given a smooth cubic surface $\mathcal{S}$ with a Schl\"afli double six $\{\ell_i\}_{i=1}^6$ and 
$\{\ell_j'\}_{i=1}^6$ that arises as the exceptional lines of $x_1, \ldots, x_6 \in \PP^2_x$ and $y_1,\ldots y_6 \in \PP^2_y$, the matrix $Z_6 = (x_i^\top \otimes y_i^\top)_{i=1}^6$ is rank deficient.
\end{cor}
\begin{proof}
Lemma \ref{lem:counterpart rational function} provides us with unique (up to homography) points $y_1',\ldots,y_6'\in\PP^2_y$ such that $Z'_6=(x_i^\top \otimes {y'}^\top_i)_{i=1}^6$ is rank deficient. Then, following along the same lines as the proof of Lemma \ref{lem:counterpart rational function}, $S=\mathcal{N}(Z'_6)\cap\mathcal{D}$ and the points $\{y_i\}$ and $\{y_i'\}$ must be related by homography. It follows by Lemma \ref{lem:rank drop is invariant under homographies} that $Z_6$ is rank deficient.
\end{proof}

Combining Corollaries \ref{cor:rank deficiency gives SD6} and \ref{cor:Sd6 gives rank deficiency} we obtain a method to compute a determinantal representation of a smooth cubic surface starting with a Schl\"afli double six on it.

\begin{cor} \label{cor:alt det rep} 
Let $\mathcal{S}$ be a smooth cubic surface  with a Schl\"afli double six, $\{\ell_i\}$ and 
$\{\ell_j'\}$, and let $\{(x_i,y_i) \}_{i=1}^6 \subset \PP^2 \times \PP^2$ be such that 
$\ell_i$ are the exceptional lines of the blow up of $\PP^2_x$ at $x_1, \ldots, x_6$ and $\ell_6'$ are the exceptional lines of the blow up of $\PP^2_y$ at $y_1,\ldots y_6$. Then a determinantal representation of $\mathcal{S}$ is given by $M(z) = \sum z_iM_i$ where 
$M_0,\ldots,M_3$ is a basis of the null space of $Z_6 = (x^\top_i \otimes y^\top_i)_{i=1}^6$.
\end{cor}
This gives us the final statement of Theorem \ref{thm:k=6 central theorem}, concluding the proof.

\begin{rem}In the classical literature one finds a determinantal representation of a cubic surface from $6$ points in a $\PP^2$ (whose blow up at these points is the surface), via the Hilbert-Burch theorem (see Appendix~\ref{sec:cubic surface facts}). It is worthwhile to contrast that method with that in Corollary~\ref{cor:alt det rep} where the starting point is 
$6$ point pairs in $\PP^2 \times \PP^2$. Then a determinantal representation \eqref{eq:matrix pencil} of the surface can be obtained from a basis of $\mathcal{N}(Z_6)$ using simple linear algebra.
\end{rem}

\subsubsection{Conic intersections and the bracket ideal $\mathcal{B}_{rackets}(Z_6)$} 

We now use Theorem \ref{thm:k=6 central theorem} to explain the ideal $\mathcal{B}_{rackets}(Z_6)$ used in the algebraic characterization of rank deficiency in Theorem \ref{thm:k=6 neither side in a line}.

Consider a Schl\"afli double six $\{\ell_i\}_{i=1}^6,\{\ell_j'\}_{j=1}^6$ on a cubic surface $\mathcal{S}$ where the lines $\ell_i$ are the exceptional lines under the blow up $\pi_x$ of points $x_i \in \PP^2_x$ and the lines $\ell'_j$ are the exceptional lines under the blow up $\pi'_y$ of points $y_j \in \PP^2_y$. Let $C_i$ be the unique conic passing through $\{x_1, \ldots ,x_6\}\backslash\{x_i\}$ and $C'_i$ the unique conic passing through $\{y_1, \ldots ,y_6\}\backslash\{y_i\}$. By Lemma \ref{lem:cubic surface facts 1}, the lines $\ell_i$ are the strict transforms of the conics $C_i'\subset\PP^2_y$ and and the lines $\ell_j'$ are the strict transforms of the conics $C_j\subset\PP^2_x$. There is then a birational automorphism $\PP^2_x\to S\to\PP^2_y$ such that $x_i\mapsto C_i'$ and $C_i\mapsto y_i$. This map is a degree $5$ Cremona transformation and can be formalized using the following recipe due to Rudolf Sturm \cite{sturm}, mentioned by Werner in \cite{wernerfivepoints}. These 
conics can be used to give a third construction of the unique 
$6$th point pair $(x_6,y_6)$ that can be added to 
$\{(x_i,y_i)\}_{i=1}^5$ in general position so that $Z_6$ is rank deficient. We state the result without proof in Lemma~\ref{lem:Werner conics} and illustrate it in 
Figure~\ref{fig:Werner conics}. The data for this figure was chosen to make the conics easily visible. 

\begin{lem} \cite{sturm} \label{lem:Werner conics}
Let $C$ be the unique conic through $x_1, \ldots, x_5$ in general position, and let $H_{i}$ be the unique 
homography such that $y_j = H_i x_j$ for all $j \neq i$. Define $C_i' := H_i(C)$ for $i=1,\ldots,5$. Then 
$y_6 \in \bigcap_{i=1}^5 C_i'$ is the unique point in the intersection of the conics $C_i'$.

Similarly, let $C'$ be the unique conic through $y_1, \ldots, y_5$, and let $G_{j}$ be the unique homography such that $x_i = G_j y_i$ for all $i \neq j$. Define $C_j := H_j(C')$ for $j=1,\ldots,5$. Then $x_6 \in \bigcap_{j=1}^5 C_j$ is the unique point in the intersection of the conics $C_j$.
\end{lem}

\begin{figure}
    \centering
    \fbox{%
    \includegraphics[scale=0.175]{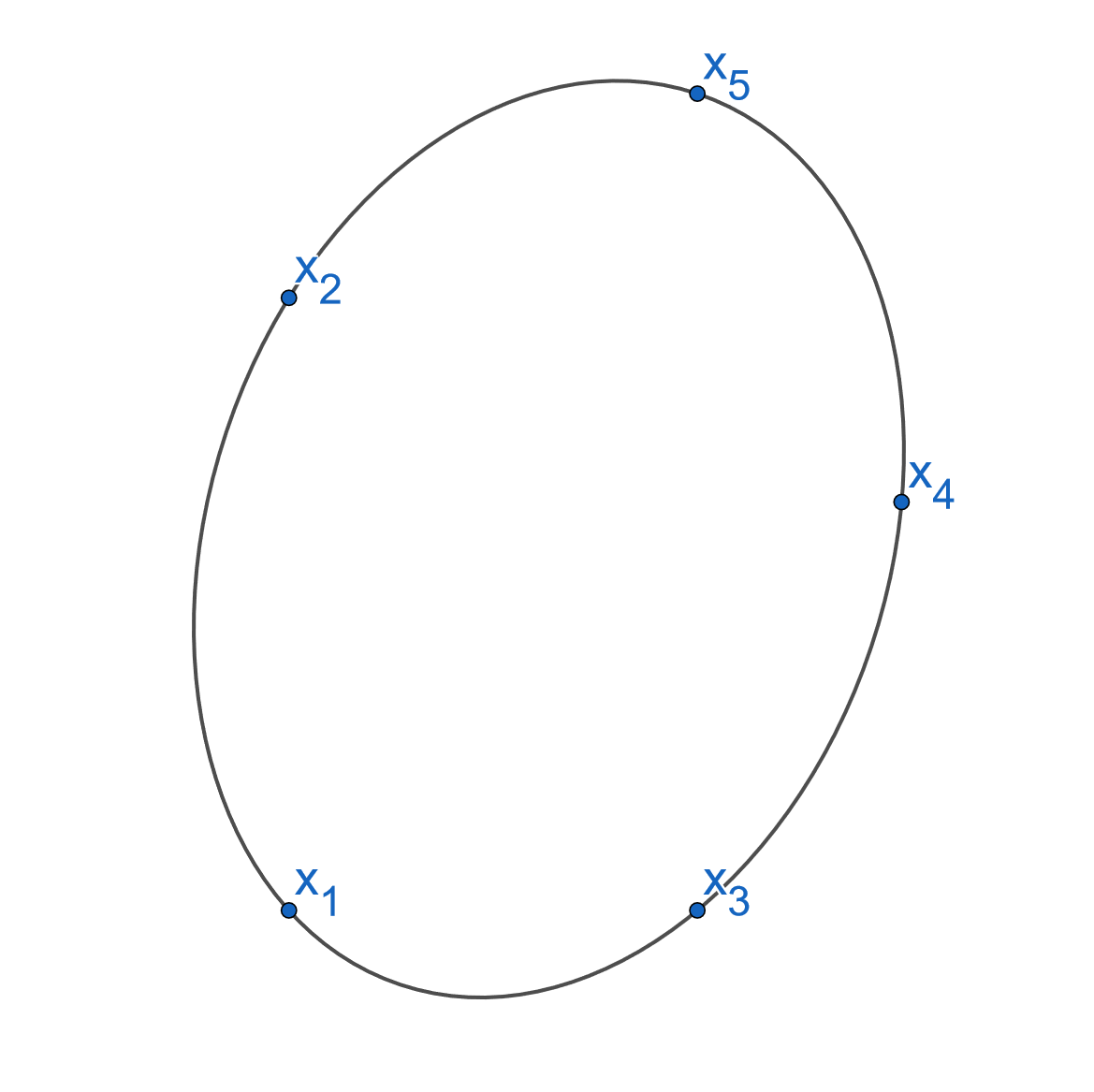}
    }
    \fbox{%
    \includegraphics[scale=0.113]{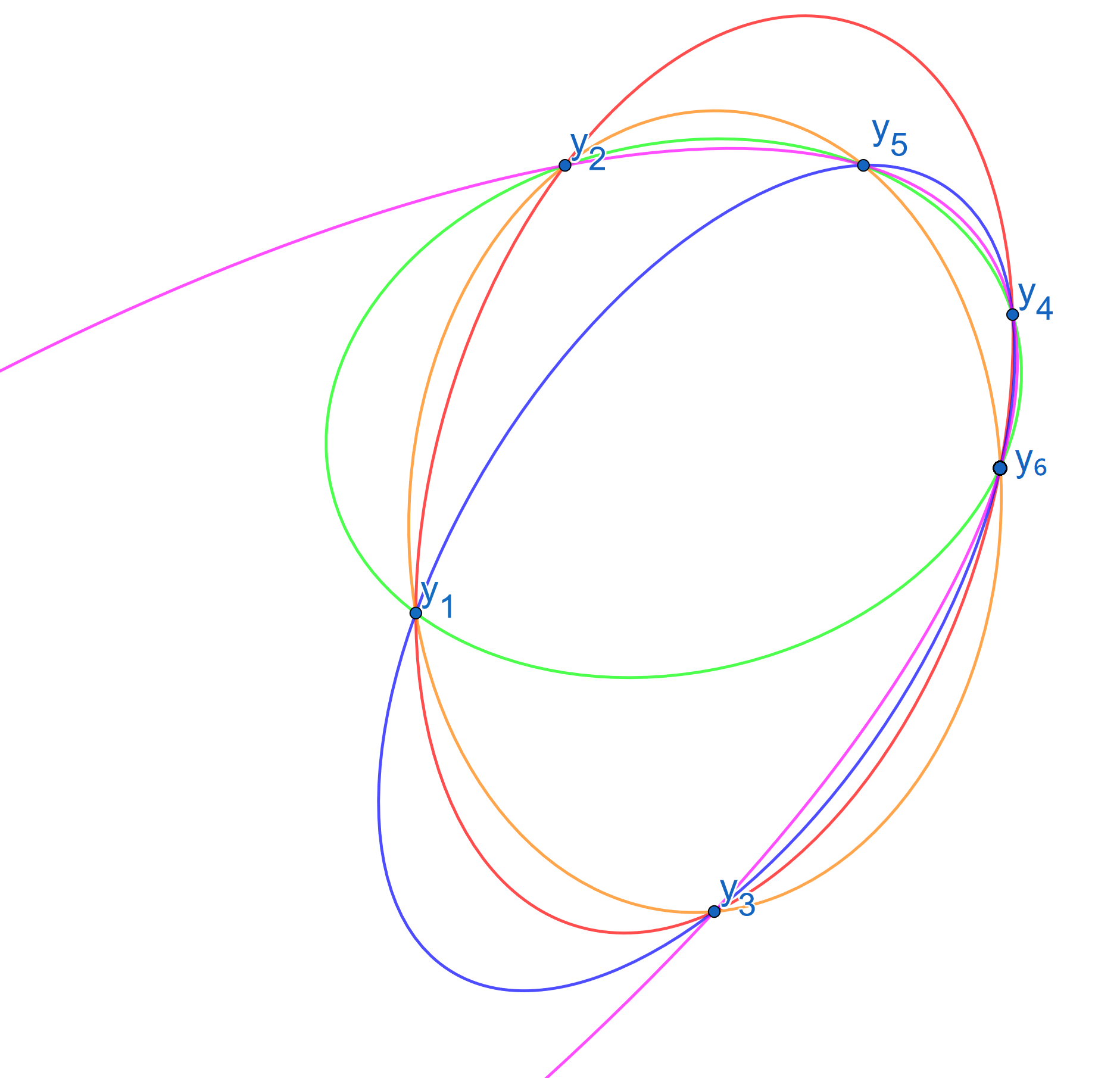}
    }
    \caption{The point $y_6$ is the unique intersection point of the $5$ conics $C_i'$ in $\PP^2_y$. \label{fig:Werner conics}}
\end{figure}

We can now restate Theorem~\ref{thm:k=6 central theorem} using Lemma~\ref{lem:Werner conics} to yield another characterization of minimal rank deficiency of $Z_6$.  
\begin{thm} \label{thm:Werner conics}
Let $\{(x_i,y_i)\}_{i=1}^5$ be in general position. Then 
$(x_6,y_6)$ is the unique new point under which $Z_6$ is rank deficient  if and only if $y_6$ is the unique intersection point $\cap_{j=1}^5 C_i'$ and $x_6$ is the unique intersection point $\cap_{j=1}^5 C_j$. 
\end{thm}

Theorem~\ref{thm:Werner conics} explains the bracket and cross-ratio conditions that appear in Theorem~\ref{thm:k=6 neither side in a line}. 
Let $C$ be the conic through $x_1,\ldots,x_5$ and let $H$ be the homography such that $Hx_i=y_i$ for all $i=1,2,3,4$. Then $C':=HC$ is a conic that 
passes through $y_1 = Hx_1, y_2 = Hx_2, y_3 = Hx_3, y_4 = Hx_4$ and 
$Hx_5$. By Lemma~\ref{lem: conic formula},  $y_6\in C'$ if and only if
\begin{equation}
\det(H)[135]_x\det(H)[245]_x[146]_y[236]_y=[136]_y[246]_y\det(H)[145]_x\det(H)[235]_x.
\end{equation}
This can be rewritten as the bracket 
equation \eqref{eq:k=6 brackets} by eliminating $\det(H)$:
\begin{align}
[135]_x[245]_x[146]_y[236]_y&=[136]_y[246]_y[145]_x[235]_x
\end{align}
The other 29 bracket equations follow by selecting different conics and homographies.

\subsubsection{The Cremona hexahedral form of a cubic surface}
We now give a third characterization of minimal rank deficiency of $Z_6$ via the {\em Cremona hexahedral form} of a cubic surface. This will explain the statement of Theorem~\ref{thm:k=6 general position}. 
\medskip

Recall that every smooth cubic surface is the blow up of $\PP^2$ at $6$ points in general position. The Cremona hexahedral form of the surface provides explicit equations for it in terms of the points being blown up. We briefly describe these equations and the properties we need. More details can be found in \cite{coble}*{Section 4}, \cite{dolgachev}*{Section 9.4.3}, \cite{Ottaviani2013FiveLO}*{Section 5.6}. 

Given $6$ points $p_1,\ldots,p_6 \in \PP^2$, consider the 
following list of $6$ cubic polynomials in $u=(u_0,u_1,u_2)$, the coordinates of $\PP^2$, that vanish on them. 
\begin{small}
\begin{align} \label{eq:covariant cubics}
\begin{split}
a(u)&=[25u][13u][46u]+[51u][42u][36u]+[14u][35u][26u]+[43u][21u][56u]+[32u][54u][16u]\\
b(u)&=[53u][12u][46u]+[14u][23u][56u]+[25u][34u][16u]+[31u][45u][26u]+[42u][51u][36u]\\
c(u)&=[53u][41u][26u]+[34u][25u][16u]+[42u][13u][56u]+[21u][54u][36u]+[15u][32u][46u]\\
d(u)&=[45u][31u][26u]+[53u][24u][16u]+[41u][25u][36u]+[32u][15u][46u]+[21u][43u][56u]\\
e(u)&=[31u][24u][56u]+[12u][53u][46u]+[25u][41u][36u]+[54u][32u][16u]+[43u][15u][26u]\\
f(u)&=[42u][35u][16u]+[23u][14u][56u]+[31u][52u][46u]+[15u][43u][26u]+[54u][21u][36u]
\end{split}
\end{align}
\end{small}
In fact, each summand in each cubic is the product of the lines through $3$ disjoint pairs of the $6$ points and already vanishes on the points. Formally these expressions are similar to 
the scalars in \eqref{eq:bumped up Joubert invariants}. These cubic polynomials are {\em covariants} of $p_1, 
\ldots, p_6$ under the action of $\textup{PGL}(3)$, and appear in the work of Coble 
\cite{coble}. 

 The $6$ cubics in \eqref{eq:covariant cubics} span the $4$-dimensional vector space of cubics that vanish on $p_1,\ldots, p_6$ and can be used to blow up 
 $\PP^2$ at $p_1, \ldots, p_6$ to obtain a cubic surface ${S}$. This surface is the closure of 
\begin{align}\label{eq: k=6 explicit cubic equations for canonical blowup}
    \{(a(u),b(u),c(u),d(u),e(u),f(u)) ~:~u\in\PP^2\})\subset\PP^5, 
\end{align}
 and the zero set of the {\bf Cremona hexahedral equations}:
\begin{align}\label{eq:canonical blow up}
\begin{split}
    z_1^3+z_2^3+z_3^3+z_4^3+z_5^3+z_6^3&=0\\
    z_1 + z_2 + z_3 + z_4 + z_5 + z_6 & = 0\\
    \bar{a} z_1+\bar{b} z_2 +\bar{c} z_3+\bar{d}z_4 +\bar{e} z_5 +\bar{f} z_6 &=0 
    \end{split}
\end{align}
where the scalars $\bar{a},\ldots,\bar{f}$ are as in 
\eqref{eq:bumped up Joubert invariants}. Here, $z_1, \ldots, z_6$ as the coordinates of $\PP^5$.

The {\bf Cremona hexahedral form} of $S$ can be used to find the $27$ lines on $S$ explicitly via  expressions for the $45$ tritangent planes of $S$, see \cite{coble}*{Section 4}. In particular the $15$ lines that are not 
part of the Schl\"afli double six corresponding to the blow up are obtained by permuting the equations:
\begin{equation}\label{eq:15 bad lines}
z_1+z_4=z_2+z_5=z_3+z_6=0.
\end{equation}
Under the blow up morphism $\pi:S\to\PP^2$, these lines correspond to specific lines $\overline{p_ip_j}$; the one above corresponds to the line between $p_1$ and $p_2$.

We can characterize the rank deficiency of $Z_6$ via the Cremona hexahedral form.

\begin{thm}\label{thm:canonical blow up}
Let $x_1,\ldots,x_6\in\PP^2_x$ and $y_1,\ldots,y_6\in\PP^2_y$ be in general position and let $S_x$ (respectively, $S_y$) be the blow up of $\PP^2_x$ (respectively, $\PP^2_y$) at $x_1, \ldots, x_6$ (respectively, $y_1, \ldots, y_6$) in Cremona hexahedral form.  Then $S_x=S_y$ if and only if either there exists homography $H$ such that $Hx_i=y_i$ for all $i=1,\ldots,6$ or the matrix $Z_6$ is rank deficient.
\end{thm}

\begin{proof}
If $S_x=S_y$, then they both contain the $15$ lines in equation \eqref{eq:15 bad lines}, none of which correspond to the proper transforms of either set of $6$ points under their respective blow up morphisms. It follows that either the exceptional lines of $x_1,\ldots,x_6$ are exactly the exceptional lines of $y_1,\ldots,y_6$ or that the two sets are entirely distinct. Moreover, since each of the $15$ lines in \eqref{eq:15 bad lines} have specific intersection conditions, either $\ell_{x_i}=\ell_{y_i}$ for all $i$, or $\{\ell_{x_1},\ldots,\ell_{x_6}\}$ and $\{\ell_{y_1},\ldots,\ell_{y_6}\}$ form a Schl\"afli double six.

In the first case, the exceptional lines of $x_1,\ldots,x_6$ are exactly the exceptional lines of $y_1,\ldots,y_6$. Then we can define an invertible projective transformation $\pi_y\circ\pi_x^{-1}:\PP^2_x\to\PP^2_y$, where $\pi_x$ and $\pi_y$ again refer to the blow up morphisms, such that $x_i\mapsto y_i$ for all $i$. This means that for each $i=1,\ldots,6$, $x_i$ and $y_i$ are related by a homography $H$.

In the second case the exceptional lines of $x_1,\ldots,x_6$ are distinct from the exceptional lines of $y_1,\ldots,y_6$. In this case they form a Schl\"afli double six  and by Theorem \ref{thm:k=6 central theorem}, $Z_6$ is rank deficient.

To prove the reverse direction, suppose $Hx_i = y_i$ for some homography $H$. Then  $S_x = S_y$ because $(\bar{a}_x, \ldots \bar{f}_x) \sim (\bar{a}_y, \ldots, \bar{f}_y)$ by the formula in \eqref{eq:bumped up Joubert invariants} which means that $S_x$ and $S_y$ have the same Cremona hexahedral form. Similarly, if $Z_6$ is rank deficient then by Theorem~\ref{thm:k=6 central theorem}, $\PP^2_x$ blown up at $\{x_i\}$ agrees with $\PP^2_y$ blown up at $\{y_i\}$ and hence $S_x = S_y$.
\end{proof}

Theorem~\ref{thm:canonical blow up} provides a geometric proof of the algebraic characterization of rank deficiency in Theorem~\ref{thm:k=6 general position}. As before, 
we subscript the cubics in \eqref{eq:covariant cubics} with $x$ or $y$ depending on whether they come from $\{x_i\}$ or $\{y_i\}$. The Cremona hexahedral equations \eqref{eq:canonical blow up} imply that the two cubic surfaces $S_x$ and $S_y$ are equal if and only if the following polynomials in $u$ (coordinates on $\PP^2_x$) and 
$v$ (coordinates on $\PP^2_y$) are identically zero. 
\begin{equation}\label{eq: k=6 cubic equations equiv 0}
\begin{split}
\bar{a}_xa_y(v)+\bar{b}_xb_y(v)+\bar{c}_xc_y(v)+\bar{d}_xd_y(v)+\bar{e}_xe_y(v)+\bar{f}_xf_y(v)\equiv 0\\
\bar{a}_ya_x(u)+\bar{b}_yb_x(u)+\bar{c}_yc_x(u)+\bar{d}_yd_x(u)+\bar{e}_ye_x(u)+\bar{f}_yf_x(u)\equiv 0.
\end{split}
\end{equation}
In other words, 
\begin{equation}
\begin{split}
\bar{a}_xz_1+\bar{b}_xz_2+\bar{c}_xz_3+\bar{d}_xz_4+\bar{e}_xz_5+\bar{f}_xz_6=0 & \quad\quad\forall z\in S_y\\
\bar{a}_yz_1+\bar{b}_yz_2+\bar{c}_yz_3+\bar{d}_yz_4+\bar{e}_yz_5+\bar{f}_yz_6 = 0 & \quad\quad\forall z\in S_x
\end{split}
\end{equation}
or equivalently, the matrix
\begin{align} \label{eq:3x6 matrix}
    \begin{bmatrix} 1 & 1 & 1 & 1 & 1 & 1\\
    \bar{a}_x &  \bar{b}_x &  \bar{c}_x &  \bar{d}_x &  \bar{e}_x &  \bar{f}_x\\
    \bar{a}_y &  \bar{b}_y &  \bar{c}_y & \bar{d}_y &  \bar{e}_y &  \bar{f}_y\end{bmatrix} 
\end{align}
is rank deficient. One can check using Macaulay2 that the ideal of $3 \times 3$ minors of 
\eqref{eq:3x6 matrix} coincides with the ideal  $\mathcal{I}_{nvariants}(Z_6)$
of $2 \times 2$ minors of \eqref{eq:2x6 matrix}.
Therefore, we conclude that $S_x$ and $S_y$ coincide if and only if 
\eqref{eq:bar vectors are multiples of each other} holds, which is the same as 
\begin{align*}
 \begin{pmatrix} \bar{a}_x &  \bar{b}_x &  \bar{c}_x &  \bar{d}_x &  \bar{e}_x &  \bar{f}_x
 \end{pmatrix} = 
 \lambda \begin{pmatrix} 
    \bar{a}_y &  \bar{b}_y &  \bar{c}_y & \bar{d}_y &  \bar{e}_y &  \bar{f}_y \end{pmatrix} 
\end{align*}
for some non-zero scalar $\lambda$.

This concludes the discussion about $6$ points $\{(x_i,y_i)\}_{i=1}^6$ in general position. 

\subsubsection{The Proof of Theorem~\ref{thm: k=6 one side in a line}} 
We can now use the tools developed in the previous subsection to prove Theorem~\ref{thm: k=6 one side in a line} which is the last statement left to prove. This is the case of one $\PP^2$ having all input points in a line; suppose this is 
$\PP^2_y$. We first show that under this assumption, \eqref{eq: k=6 cubic equations equiv 0} can be simplified using Lemma \ref{lem:simplifyng 3x3 brackets - line}. Note that 
if $y_1,\ldots,y_6$ are in a line $\ell$ then for all points $v\in\ell$  
\begin{equation}
(a_y,b_y,c_y,d_y,e_y,f_y)(v)=(0,0,0,0,0,0), 
\end{equation}
and for all points $v\in\PP^2_y \backslash\ell$, using Lemma~\ref{lem:simplifyng 3x3 brackets - line},
\begin{equation}
(a_y,b_y,c_y,d_y,e_y,f_y)(v)\sim(A_y,B_y,C_y,D_y,E_y,F_y)
\end{equation}
where $A_y,\ldots,F_y$ are  the Joubert invariants introduced in \eqref{eq:Joubert invariants}. The Joubert invariants are known to satisfy
\begin{equation}\label{eq:Basic Joubert invariant Properties}
\begin{split}
A+B+C+D+E+F=0\\
A^3+B^3+C^3+D^3+E^3+F^3=0
\end{split}
\end{equation}
(see \cite{dolgachev}*{Theorem 9.4.10}), and  any $5$ of them give a complete basis for the space of invariants for $6$ ordered points in $\PP^1$. In particular, if $p_1,\ldots,p_6\in\PP^1$ and $q_1,\ldots,q_6\in\PP^1$ are each sets of $6$ distinct points, then there exists a homography $H$ sending $p_i\mapsto q_i$ if and only if
\begin{equation}
(A_p,\ldots,F_p)\sim(A_q,\ldots,F_q).
\end{equation}

We further note that if $\{y_i\}_{i=1}^6$ are on a line then $\bar{a}_y=\cdots=\bar{f}_y=0$ by looking at the definition 
of $[(ij)(kl)(rs)]$. 
We can then simplify \eqref{eq: k=6 cubic equations equiv 0} to the single equation \eqref{eq: k=6 one side in a line basic eq} which was:
\begin{equation*}
\bar{a}_xA_y+\bar{b}_xB_y+\bar{c}_xC_y+\bar{d}_xD_y+\bar{e}_xE_y+\bar{f}_xF_y=0.
\end{equation*}
The following lemma proves the first part of 
Theorem~\ref{thm: k=6 one side in a line}. 
\begin{lem}\label{lem:k=6 one side in a line}
Let $\{(x_i,y_i)\}_{i=1}^6 \subset \PP^2 \times \PP^2$ be such that the points in one $\PP^2$ are in a line. Then $Z_6$ is rank deficient if and only if \eqref{eq: k=6 one side in a line basic eq} holds.
\end{lem}
\begin{proof}
The proof is relatively straightforward. Suppose 
$y_1,\ldots,y_6$ are in a line. Applying homographies to each side we may assume that the $x_i$'s are all finite and that $y_i=(y_{i1},1,0)$. Then 
\begin{equation}
    Z_6=\begin{bmatrix} 
    x_{11}y_{11} & x_{11} & 0 & x_{12}y_{11} & x_{12} & 0 & y_{11} & 1 & 0\\
    x_{21}y_{21} & x_{21} & 0 & x_{22}y_{21} & x_{22} & 0 & y_{21} & 1 & 0\\
    x_{31}y_{31} & x_{31} & 0 & x_{32}y_{31} & x_{32} & 0 & y_{31} & 1 & 0\\
    x_{41}y_{41} & x_{41} & 0 & x_{42}y_{41} & x_{42} & 0 & y_{41} & 1 & 0\\
    x_{51}y_{51} & x_{51} & 0 & x_{52}y_{51} & x_{52} & 0 & y_{51} & 1 & 0\\
    x_{61}y_{61} & x_{61} & 0 & x_{62}y_{61} & x_{62} & 0 & y_{61} & 1 & 0
    \end{bmatrix}
\end{equation}
and $Z_6$ is rank deficient if and only if its  only obviously non-zero $6\times 6$ minor $\det(M)=0$. Furthermore, a quick computation will give
\begin{equation}
\bar{a}_xA_y+\bar{b}_xB_y+\bar{c}_xC_y+\bar{d}_xD_y+\bar{e}_xE_y+\bar{f}_xF_y=24\det(M)
\end{equation}
yielding the desired result.
\end{proof}

Before we prove the second part of Theorem~\ref{thm: k=6 one side in a line}, we compute an example to illustrate the statement.

\begin{ex}
Consider the following points in which the $x$ points are in general position and the $y$ points are all distinct and contained in a line. This can be seen as a configuration in $\PP^2\times\PP^1$ so we fix
\begin{align*}
x_1=(0,0,1)\quad& y_1=(0,1)  \quad\quad&x_4=(1,1,1)\quad&y_4=(-4,1)\\ 
x_2=(1,0,1)\quad&y_2=(1,1) \quad\quad&x_5=(3,5,1)\quad& y_5=(8,1)\\
x_3=(0,1,1)\quad& y_3=(3,1)
\end{align*}
and leave $x_6,y_6$ as symbolic finite points. This creates a $6 \times 6$ matrix 
\begin{equation} \label{eq:6x6 Z}
Z_6=\begin{bmatrix}
0 & 0 & 0 & 0 & 0 & 1\\
1 & 1 & 0 & 0 & 1 & 1\\
0 & 0 & 3 & 1 & 3 & 1\\
-4 & 1 & -4 & 1 & -4 & 1\\
24 & 3 & 40 & 5 & 8 & 1\\
x_{61}y_{61} & x_{61} & x_{62}y_{61} & x_{62} & y_{61} & 1
\end{bmatrix}.
\end{equation}
(The reader may also view the $y_i$ as points on the line at infinity in $\PP^2$ where the third coordinate is $0$; the non-zero part of the corresponding $6 \times 9$ matrix $Z_6$ is \eqref{eq:6x6 Z}.)

The matrix $Z_6$ in \eqref{eq:6x6 Z} is rank deficient exactly when its determinant, a quadratic polynomial in the coordinates of $x_6$ and $y_6$, is $0$. If we arbitrarily fix $x_6=(2,11,1)$, then 
$\det(Z_6)=-918y_{61}+2942$, and $y_6=(2942,918)$ is the unique point $y_6$ such that the configuration is rank deficient.

If we consider only the first $5$ point pairs we fixed initially, there exists a unique projective transformation
\begin{equation}
T=\begin{bmatrix}
146 & -294 & 0\\ 135 & -109 & 11
\end{bmatrix}
\end{equation}
such that $Tx_i\sim y_i$ for all $i=1,\ldots,5$. Moreover, we can check that $Tx_6=(-2942,-918)\sim y_6$. This transformation $T$ has center $u=(-1617,-803,11888)$. This point actually has some further significance, because we can check that
\begin{equation}
(a_x(u),\ldots,f_x(u))\sim(48079,-55599,-88559,-17265,22529,90815)\sim(A_y,\ldots,F_y)
\end{equation}
\qed
\end{ex}

The following now finishes the proof of Theorem \ref{thm: k=6 one side in a line}.

\begin{proof}[Proof of Theorem~\ref{thm: k=6 one side in a line} (second part)]
Let $S$ be the blow up of $\PP^2_x$ in $x_1,\ldots,x_6$ in Cremona hexahedral form, and recall \eqref{eq:canonical blow up} and \eqref{eq:Basic Joubert invariant Properties}. It follows that \eqref{eq: k=6 one side in a line basic eq} holds if and only if $(A_y,\ldots,F_y)\in S$, which holds if and only if there exists $u\in\PP^2_x$ such that 
\begin{equation}
(a_x(u),\ldots,f_x(u))\sim(A_y,\ldots,F_y).
\end{equation}

We first show that if $Z_6$ is rank deficient then a projection exists. Let $\ell\subset\PP^2_x$ be a line such that $u\not\in\ell$. We then construct a projection $T':\PP^2_x\to\ell$ with center $u$ as in the proof of Theorem \ref{thm: k=5 geometry theorem} and Figure \ref{fig: k=5 projection construction}. This projection has $T'(x_i)=a_i$ for all $i$, where the $a_i$ are defined as in Figure \ref{fig: k=5 projection construction}. Then
\begin{equation}
(A_a,\ldots,F_a)\sim(a_x(u),\ldots,f_x(u))\sim(A_y,\ldots,F_y).
\end{equation}
Since the Joubert invariants generate the space of all invariants for $6$ ordered points on a line it follows that there exists a homography $H:\ell\to\PP^1_y$ such that $Ha_i=y_i$. We conclude that $T=HT'$ is the desired projection $T:\PP^2_x\to\PP^1_y$.

To prove the reverse direction, suppose such a projection $T$ exists. We can use homography on each side to fix $u=(0,0,1)$ and to fix the image line at infinity. Equivalently, pick homographies on each side such that
\begin{equation}
T=\begin{bmatrix}1 & 0 & 0\\ 0 & 1 & 0\end{bmatrix}.
\end{equation}
It then follows that $x_i=(y_{i1},y_{i2},b_i)$ for all $i$, where $b_i$ are unknown scalars. The matrix $Z_6$  then becomes
\begin{equation}
\begin{bmatrix}
y_{11}^2 & y_{11}y_{12} & y_{11}y_{12} & y_{12}^2 & y_{11}b_1 & y_{12}b_1\\ 
y_{21}^2 & y_{21}y_{22} & y_{21}y_{22} & y_{22}^2 & y_{21}b_2 & y_{22}b_2\\
y_{31}^2 & y_{31}y_{32} & y_{31}y_{32} & y_{32}^2 & y_{31}b_3 & y_{32}b_3\\
y_{41}^2 & y_{41}y_{42} & y_{41}y_{42} & y_{42}^2 & y_{41}b_4 & y_{42}b_4\\
y_{51}^2 & y_{51}y_{52} & y_{51}y_{52} & y_{52}^2 & y_{51}b_5 & y_{52}b_5\\
y_{61}^2 & y_{61}y_{62} & y_{61}y_{62} & y_{62}^2 & y_{61}b_6 & y_{62}b_6
\end{bmatrix}
\end{equation}
which is clearly rank deficient. This concludes the proof of Theorem~\ref{thm: k=6 one side in a line}.
\end{proof}

\bibliography{Part1}
\newpage 
\appendix 
\section{Facts about cubic surfaces}
\label{sec:cubic surface facts}
This is a collection of classical facts about cubic surfaces needed in Section~\ref{sec:k=6 geometry}.
See \cite{Hartshorne}*{V. Chapter 4} for a modern reference and \cite{coble} or \cite{Hilbert-Cohn-Vossen} for classical references.

\begin{lem}\label{lem:cubic surface facts 1}
Let $\mathcal{S} \subset \PP^3$ be a smooth cubic surface.
Then,
\begin{enumerate}
    \item $\mathcal{S}$ can be obtained by blowing up $6$ points 
    $p_1, \ldots, p_6 \in \PP^2$ in general position.
    The exceptional curves of $p_1, \ldots, p_6$ in the blow up morphism $\pi \,:\, \mathcal{S} \rightarrow \PP^2$ are $6$ lines $\ell_1, \ldots, \ell_6$ on $\mathcal{S}$. 
    
    
    \item The cubic surface $\mathcal{S}$ has exactly $27$ (complex) lines. They are: 
    \begin{enumerate}
        \item the exceptional lines $\ell_1, \ldots,\ell_6$,
        \item the strict transforms $\ell_{ij}$ of the lines in $\PP^2$ passing through $p_i,p_j$ for $1 \leq i < j \leq 6$ ($15$ of them), and 
        \item the strict transforms $\ell_j'$ of the conics in $\PP^2$ passing through five of the $p_i$ for $i \neq j, j=1,\ldots,6$ ($6$ of them). 
    \end{enumerate}
    
     \item The two sets of lines $\{\ell_1, \ldots, \ell_6\}$ and 
     $\{\ell_1', \ldots, \ell_6'\}$ 
     form a Sch\"afli double six on $\mathcal{S}$. 
     There are $36$ Schl\"afli double sixes on $\mathcal{S}$. Each double six is uniquely determined by three of the line pairs $(\ell_i,\ell_i')$. 
     
    \item 
    Any set of $6$ mutually skew lines among the $27$ lines on $\mathcal{S}$ can play the role of 
    $\ell_1, \ldots, \ell_6$. Precisely, if $\{g_1, \ldots, g_6\}$ is a set of $6$ mutually skew lines among the $27$ lines, then there is another  morphism 
    $\pi' \,:\, \mathcal{S} \rightarrow \PP^2$, making $\mathcal{S}$ isomorphic to $\PP^2$ with six points $q_1, \ldots, q_6$ in general position blown up  such that $g_1, \ldots, g_6$ are the exceptional curves of 
    $q_1,\ldots,q_6$ under  $\pi'$. 
\end{enumerate}
\end{lem}



We now note several facts about double sixes in $\PP^3$ (or more generally line configurations in $\PP^3$). 

\begin{lem}\label{lem:double six facts}
\begin{enumerate}
    \item Each double six in $\PP^3$ is contained in a unique non-singular cubic surface and thus forms part of the set of $27$ lines on that surface. 
    \item Any set of $6$ mutually skew lines among $27$ lines that satisfy the incidence relations is one half of a unique Schl\"afli double six and determines the other lines and the cubic surface.
\end{enumerate}
\end{lem}

Next we talk about determinantal representations of a cubic surface.

\begin{defn} \label{def:det rep}
A {\bf determinantal representation} of a cubic surface $\mathcal{S}$ defined by a cubic equation $F(z) =F(z_0,z_1,z_2,z_3) = 0$ is a $3\times 3$ matrix of linear forms 
\begin{align}
  M(z) = M(z_0,z_1,z_2,z_3)=z_0M_0+z_1M_1+z_2M_2+z_3M_3 
\end{align} 
such that $\det(M(z))=\lambda F(z)$ for some $\lambda\neq 0$. Two representations $M(z)$ and $M'(z)$ are equivalent if $M'(z) = AM(z)B$ for $A,B \in \textup{GL}(3)$.
\end{defn}
Note that if $M(z)$ is a determinantal representation of $\mathcal{S}$, then so is $M(z)^\top$.
We present the following facts about determinantal representations of a cubic surface 
\cite{buckley_kosir}. 

\begin{lem}\label{lem:determinantal representation facts}
\begin{enumerate}
    \item Every smooth cubic surface $\mathcal{S}$ has a determinantal representation. 
    \item The equivalence classes of determinantal representations of $\mathcal{S}$ are in bijection with sets of $6$ mutually skew lines on $\mathcal{S}$ that are the 
    exceptional curves of the blow up of $\PP^2$ at the $6$ points that gives $\mathcal{S}$. 
    
    \item If the determinantal representation $M$ of $\mathcal{S}$ corresponds to the skew lines $\ell_1, \ldots, \ell_6$ on $\mathcal{S}$, then $M^\top$ corresponds to another $6$ skew lines 
    $\ell_1', \ldots, \ell_6'$ on $\mathcal{S}$ such that the pair of $6$ lines form a Schl\"afli double six on $\mathcal{S}$. Thus there are 
    $72$ equivalence classes of determinantal representations of $\mathcal{S}$ one for each half of the $36$ Schl\"afli double sixes on $\mathcal{S}$.
    \item The choice of a Schl\"afli double six $\ell_1,\ldots,\ell_6,\ell_1',\ldots,\ell_6'$ on $\mathcal{S}$ is equivalent to the choice of a pair of equivalence classes of determinantal representations of $\mathcal{S}$, namely$[M]$ and $[M^\top]$. (A smooth cubic surface cannot have a symmetric determinantal representation and hence $[M] \neq [M^\top]$.) 
\end{enumerate}
\end{lem}

Here are various recipes to explicitly realize some of the above assertions. 
\begin{enumerate}
    \item If $\mathcal{S}$ is the blow up of $\PP^2$ at the $6$ points $p_1, \ldots, p_6$, then $\mathcal{S}$ can be constructed as follows. Pick 
a basis $\{f_0,f_1,f_2,f_3\}$ of the $4$-dimensional vector space of cubics through $p_1, \ldots, p_6$ and map 
\begin{align} \label{eq:usual blowup} 
    u \in \PP^2 \mapsto (f_0(u):f_1(u):f_2(u):f_3(u)) \in \PP^3. 
    \end{align}

\item The variety of $\langle f_0,\ldots, f_3 \rangle$ is $\{p_0,\ldots,p_6\}$. By the {\em Hilbert-Burch theorem} \cite{eisenbud-book}*{Section 20.4} there is a $3 \times 4$ matrix $L(u)$ of linear forms in 
$u_0,u_1,u_2$ whose maximal minors are $f_0,\ldots,f_3$. The matrix $L(u)$ generically has 
rank $3$ and has rank $2$ exactly when $u=p_i$ for some $i=1,\ldots,6$. Define a $3 \times 3$ matrix 
$M(z)$ of linear forms in $z_0,z_1,z_2,z_3$ by the relation 
\begin{align}\label{eq:HilbertBurch}
    M(z) \cdot \begin{pmatrix} u_0\\u_1\\u_2 \end{pmatrix} = L(u) \cdot \begin{pmatrix} z_0\\z_1\\z_2\\z_3 \end{pmatrix}.
\end{align}
Then $M(z)$ is a determinantal representation of $\mathcal{S}$. 

\item The blow up of $\PP^2$ at $p_1, \ldots, p_6$ can be seen via \eqref{eq:HilbertBurch}.
If $\bar{u} \in \PP^2$ is not one of $p_1, \ldots, p_6$, then some $f_i(\bar{u}) \neq 0$ and  
$\rank(L(\bar{u}) ) = 3$. Therefore, the null space of $L(\bar{u})$ is a unique point 
$\bar{z} \in \PP^3$ and $M(\bar{z})\bar{u} = L(\bar{u})\bar{z}=0$. This means that $\det(M(\bar{z}))=0$ and $\bar{z} \in \mathcal{S}$. 
If $\bar{z}$ is a generic point on $\mathcal{S}$, then $\rank(M(\bar{z})) = 2$, and the blow up morphism $\pi$ sends $\bar{z}$ to the unique point $\bar{u} \in \PP^2$ that generates the right null space of $M(\bar{z})$.

\item The $6$ exceptional curves 
of the blow up are precisely 
\begin{align}
    \ell_i := \{ z \in \PP^3 \,:\, M(z) p_i = 0 \} \textup{ for } i = 1, \ldots, 6.
\end{align}
 By what is stated above, 
$\rank(L(u))=2$ if and only if $u=p_i$ for some $i=1,\ldots,6$ if and only if (via 
\eqref{eq:HilbertBurch}) $M(z)u=L(u)z=0$ is a system of $2$ linearly independent equations in 
$z$ (i.e., defines a line in $\PP^3$). This line is entirely on $\mathcal{S}$ since 
$\det(M(z))=0$ if $M(z)p_i = 0$. 

\item Corresponding results hold for the 
determinantal representation  $M(z)^\top$ of $\mathcal{S}$.
\begin{enumerate}
\item To obtain the points $q_1, \ldots, q_6 \in \PP^2$ whose blow up is $\mathcal{S}$, compute 
$L'(u)$ via 
\begin{align}\label{eq:Hilbert-Burch-2}
    M(z)^\top u = L'(u) z.
\end{align}
The variety of the maximal minors of $L'(u)$ is $\{q_1, \ldots, q_6\}$. 
\item A point $\bar{y} \in \PP^2$ not equal to any 
$q_i$ is sent to the unique point $\bar{z}$ in 
$\PP^3$ in the right nullspace of $L'(\bar{y})$. This means 
$L'(\bar{y})\bar{z} = M(\bar{z})^\top \bar{y} = 0$ and $\bar{z} \in \mathcal{S}$ since 
$\det(M(\bar{z}))^\top = 0$.

\item The morphism $\pi' \,:\, \mathcal{S} \rightarrow \PP^2$ sends $\bar{z}$ to the unique point 
$\bar{y} \in \PP^2$ that 
generates the right nullspace of $M(\bar{z})^\top$, or equivalently, the 
left nullspace of $M(\bar{z})$.

\item The exceptional curves of the blow up are the lines 
\begin{align*}
    \ell_j' := \{ z \in \PP^3 \,:\, q_j^\top M(z) = 0 \} \textup{ for } j = 1, \ldots, 6.
\end{align*}
\end{enumerate} 
\end{enumerate}

The following theorem follows from the above facts.
\begin{thm} \label{thm:summary of cubic surface facts}
Given a determinantal representation $M(z)$ of 
a smooth cubic surface $\mathcal{S} \subset \PP^3$, there exists $6$ points $(p_i,q_i) \in \PP^2 \times \PP^2$ 
such that $\mathcal{S}$ is the blow up of the first $\PP^2$ at $\{p_i\}_{i=1}^6$ and also the blow up of the second $\PP^2$ at $\{q_j\}_{j=1}^6$. The exceptional 
curves of the two blow ups form a Schl\"afli double six on $\mathcal{S}$:
\begin{align}
 \ell_i := \{ z \in \PP^3 \,:\, M(z) p_i = 0 \} \textup{ for } i = 1, \ldots, 6,\\
     \ell_j' := \{ z \in \PP^3 \,:\, q_j^\top M(z) = 0 \} \textup{ for } j = 1, \ldots, 6.
\end{align}
\end{thm}

\end{document}